\documentclass[dvipsnames,a4paper,twoside]{article}
\usepackage{amsthm,amsmath,amssymb,xcolor,soul,xspace,bm}
\usepackage[bookmarks=true,bookmarksopen=true]{hyperref}

\usepackage[ruled,algo2e]{algorithm2e}

\theoremstyle{plain}
\newtheorem{theorem}{Theorem}[section]
\newtheorem{lemma}[theorem]{Lemma}
\newtheorem{corollary}[theorem]{Corollary}

\theoremstyle{definition}

\newtheorem{example}[theorem]{Example}
\newtheorem{assumption}[theorem]{Assumption}

\theoremstyle{remark}
\newtheorem{remark}[theorem]{Remark}

 \numberwithin{equation}{section}
\numberwithin{table}{section}

\title{Quadratic convergence of an SQP method for some optimization problems with applications to control theory\thanks{The authors were supported by MICIU/AEI/10.13039/501100011033/ under research project PID2023-147610NB-I00.}}

\author{Eduardo Casas\thanks{Departamento de Matem\'{a}tica Aplicada y Ciencias de la Computaci\'{o}n, E.T.S.I. Industriales y de Telecomunicaci\'on, Universidad de Cantabria, 39005 Santander, Spain
(\texttt{eduardo.casas@unican.es})},
\and
Mariano Mateos\thanks{Departamento de Matem\'{a}ticas, Campus de Gij\'on, Universidad de Oviedo, 33203, Gij\'on, Spain (\texttt{mmateos@uniovi.es})}
}

\ifpdf
\hypersetup{
  pdftitle={Quadratic convergence of an SQP method for some optimization problems with applications to control theory},
  pdfauthor={E.~Casas, and M.~Mateos}
}
\fi

\newcommand{\dx}{\,\mathrm{d}x}
\newcommand{\dt}{\,\mathrm{d}t}
\newcommand{\dmu}{\,\mathrm{d}\mu}

\newcommand{\Pb}{\mbox{\rm (P)}\xspace}
\newcommand{\Pbu}{\mbox{\rm (P$_1$)}\xspace}
\newcommand{\Pbd}{\mbox{\rm (P$_2$)}\xspace}
\newcommand{\Pbt}{\mbox{\rm (P$_3$)}\xspace}
\newcommand{\Pbc}{\mbox{\rm (P$_4$)}\xspace}
\newcommand{\uad}{U_{\rm ad}}
\newcommand{\dimension}{d}
\newcommand{\conormal}{n\!}
\newcommand{\tikhonov}{\kappa}
\newcommand{\mA}{\mathcal{A}}

\newcommand{\mJ}{\mathcal{J}}

\newcommand{\mQ}{\mathcal{Q}}

\newcommand{\taup}{\tau^+}
\newcommand{\taum}{\tau^-}

\newcommand{\proj}{\operatorname{Proj}}
\newcommand{\umin}{{\alpha}}
\newcommand{\umax}{{\beta}}

\begin{document}

\maketitle

\begin{abstract}
We analyze a sequential quadratic programming algorithm for solving a class of abstract optimization problems. Assuming that the initial point is in an $L^2$ neighborhood of a local solution that satisfies no-gap second-order sufficient optimality conditions and a strict complementarity condition, we obtain stability and quadratic convergence in $L^q$ for all $q\in[p,\infty]$ where $p\geq 2$ depends on the problem. Many of the usual optimal control problems of partial differential equations fit into this abstract formulation. Some examples are given in the paper. Finally, a computational comparison with other versions of the SQP method is presented.
\end{abstract}

\begin{quote}
\end{quote}

\begin{quote}
\end{quote}

\section{Introduction}\label{S1}
\setcounter{equation}{0}
This paper is devoted to the analysis of an SQP method for the abstract optimization problem
\[
\Pb\qquad\min_{u\in\uad} J(u): = \mathcal{J}(u)+\frac{\tikhonov}{2}\|u\|_{L^2(X)}^2.
\]
Assumptions on the sets and functionals involved are made in a way that allows us to apply the results to a wide variety of control-constrained optimal control problems.

To our knowledge, the papers \cite{Goldberg-Troltz1998,Troltz1999} are among the first to study SQP methods for solving control-constrained optimal control problems governed by partial differential equations. Although the results presented in these papers are not directly applicable to other optimal control problems, the proof technique has been successfully applied to problems governed by different PDEs; see e.g. \cite{Wachsmuth2007,Hoppe-Neitzel2021,Hehl-Neitzel2024,Ammann-Yousept2025}. In all these references, the authors describe and study a method to solve the complete optimality system of state equation, adjoint state equation, and variational inequality using the three variables state, adjoint state, and control. This is sometimes called the Lagrange-Newton method. Their proofs of the quadratic convergence of the algorithm are based on a strong second-order sufficient optimality condition satisfied by the optimal control. In contrast, we will analyze an SQP method that focuses on the variational inequality. We use only the control as the optimization variable, while the state and adjoint state are taken as functions of the control. Our proof of the quadratic convergence of the algorithm is based on a no-gap second-order sufficient optimality condition and a strict complementarity condition, which are the usual assumptions in the case of finite dimensional optimization problems. Following \cite{Alt1994} or \cite{Casas2024}, we analyze the method for an abstract problem and then apply the results to some optimal control problems of partial differential equations, including problems governed by semilinear elliptic and parabolic equations, and Navier-Stokes equations, with boundary or distributed controls acting either in an additive or multiplicative (bilinear) way.

The outline of the paper is as follows. The assumptions and some results about \Pb are given in Section \ref{S2}, and the convergence analysis of the SQP is done in Section \ref{ZZ-S3}.  In Section \ref{ZZ-S4}, we show several examples of control-constrained optimal control problems that fit into our abstract framework. Finally, in Section \ref{ZZ-S5} the method presented in this paper is compared with the Lagrange-Newton method for two of the proposed problems. Numerical experience shows that our method can be more efficient for some problems, mainly in the case of evolutionary partial differential equations.

\section{An abstract optimization problem}\label{S2}
\setcounter{equation}{0}
Let $(X,\mathcal{S},\mu)$ be a measure space with $\mu(X)<\infty$. Given $r\geq 1$, $\rho>0$, and $u\in L^r(X)$, we denote $B^r_\rho(u) = \{v\in L^r(X):\ \|v-u\|_{L^r(X)}\leq\rho\}$.

Let $\mathcal{J}:L^p(X) \longrightarrow \mathbb{R}$ be a function of class $C^2$ for some $p\in[2,+\infty]$, $\tikhonov>0$, and  $-\infty\leq\umin<\umax\leq +\infty$. If $p>2$, we also require $-\infty <\umin<\umax< +\infty$. Our goal is to analyze the SQP method to solve the following optimization problem
\[
\Pb\qquad\min_{u\in\uad} J(u): = \mathcal{J}(u)+\frac{\tikhonov}{2}\|u\|_{L^2(X)}^2,
\]
where $\uad = \{u\in L^p(X):\ \umin\leq u\leq \umax \text{ a.e. }[\mu]\}$. As in \cite{Casas2024}, we assume that there exists an open set $\mathcal{A}\subset L^p(X)$ such that $\uad\subset\mathcal{A}$, $J:\mathcal{A}\to\mathbb{R}$ is of class $C^2$, and the following assumptions hold:

\begin{assumption}\label{A2.1} There exists a $C^1$ mapping $\Phi:\mathcal{A}\to L^\infty(X)$ such that
\[\mathcal{J}'(u)v = \int_X \Phi(u) v\dmu\ \forall u\in \mathcal{A}\text{ and }\ \forall v\in L^p(X).\]
\end{assumption}
We note for later reference that since $\Phi$ is $C^1$, $\Phi$ is also locally Lipschitz. This means that for every $u\in\mathcal{A}$ there exist $\rho_u>0$ with $B^p_{\rho_u}(u)\subset\mathcal{A}$ and constants $M_{u,\Phi'}>0$ and $L_{u,\Phi}>0$ such that
\begin{align}
\|\Phi(w)-\Phi(u) \|_{L^\infty(X)} \leq & L_{u,\Phi} \|w-u\|_{L^p(X)}&& \forall w \in B^p_{\rho_u}(u),
\label{E2.1}\\
 \|\Phi'(w)v\|_{L^\infty(X)}\leq & M_{u,\Phi'}\|v\|_{L^p(X)}&&
\forall w\in B^p_{\rho_u}(u) \ \text{ and }\ \forall v\in L^p(X).
\label{E2.2}
\end{align}
Since $\mJ$ is of class $C^2$, we have that for every $u\in\mA$, $\mJ''(u): L^p(X)\times L^p(X) \longrightarrow \mathbb{R}$ is a symmetric and continuous bilinear form which satisfies the following identity:
\begin{equation}\label{E2.3}
\mJ''(u)(v_1,v_2)=\int_X[\Phi'(u)v_1]v_2\dmu = \int_X[\Phi'(u)v_2]v_1\dmu\quad \forall v_1,v_2\in L^p(X).
\end{equation}
As usual, we will write $\mJ''(u)v^2=\mJ''(u)(v,v)$.

\begin{assumption}\label{A2.2}
For every $u\in\mathcal{A}$ the linear mapping $\Phi'(u):L^p(X)\longrightarrow L^\infty(X)$ has an extension to a compact operator
$\Phi'(u):L^2(X)\longrightarrow L^2(X)$. Furthermore, for every $\varepsilon > 0$ there exists $\rho = \rho_{\varepsilon,u} >0$ with $B^p_\rho(u)\subset\mA$ such that
\begin{equation}
\label{E2.4}
\|(\Phi'(u)-\Phi'(w))v\|_{L^2(X)}\leq \varepsilon \|v\|_{L^2(X)}\ \forall w\in B_\rho^p(u)\text{ and }\forall v\in L^2(X).
\end{equation}

\end{assumption}
This assumption implies that for all $u\in\mA$ we can extend $\mJ''(u)$ to a continuous bilinear form on $L^2(X)\times L^2(X)$ defined by the expression \eqref{E2.3}.

\begin{assumption}\label{A2.3}There  exist $N\geq 0$ and numbers $2=p_0 \leq p_1 \leq \ldots \leq p_N = p$ such that
for every $u\in\mathcal{A}$, the linear mapping $\Phi'(u):L^p(X)\longrightarrow L^\infty(X)$ defines also a continuous operator $\Phi'(u):L^{p_{i-1}}(X)\longrightarrow L^{p_i}(X)$ for $i=1,\ldots,N$. For ease of writing, we will set $p_{N+1} = +\infty$.
\end{assumption}
We will further suppose that the derivative of $\Phi$ is Lipschitz continuous in the following sense.

\begin{assumption}\label{A2.4}For every $u\in\mathcal{A}$, there exists $\rho_u>0$ with $B^p_{\rho_u}(u)\subset\mathcal{A}$ and a constant $L_{u,\Phi'}$ such that
\begin{equation}\label{E2.5}
  \| ( \Phi'(w)-\Phi'(u) \big) v\|_{L^\infty(X)} \leq L_{u,\Phi'} \|w-u\|_{L^{p}(X)} \|v\|_{L^p(X)}
\end{equation}
for all $(w,v) \in B^p_{\rho_u}(u) \times L^{p}(X)$.
\end{assumption}

\begin{remark}\label{R2.5}
Assumption \ref{A2.1} is the same as \cite[Hypothesis (H1)]{Casas2024}. Assumptions \ref{A2.2} and \ref{A2.3} generalize \cite[Hypothesis (H2)]{Casas2024}.
Assumption \ref{A2.3} may seem a little strange, because in many cases it is satisfied with $N = 0$ or $N = 1$. However, the reader is referred to the example given in Section \ref{S3.5} for a control problem where we need $N = 3$.
\end{remark}

Using the finiteness of $\umin$ and $\umax$ when $p>2$, we can conclude that for $p < \infty$, $\bar u$ is a local minimizer of \Pb in the sense of $L^p(X)$ if and only if it is also a local minimizer in the sense of $L^2(X)$. In the rest of the paper, local minimizers will be understood in the sense of $L^2(X)$. Every local minimizer $\bar u\in \uad$ satisfies
\begin{equation}\label{E2.6}
  \int_X (\Phi(\bar u) + \tikhonov \bar u)(u-\bar u)\dmu \geq 0\quad \forall u\in\uad.
\end{equation}
The following pointwise expression is a well known consequence of \eqref{E2.6}.
\begin{equation}\label{E2.7}
\bar u(x) =\proj_{[\umin,\umax]}\left[-\frac{1}{\tikhonov}\Phi(\bar u)(x)\right]\text{ for a.a.}\, [\mu]\,  x\in X.
\end{equation}
This identity and Assumption \ref{A2.1} imply that any local minimizer $\bar u\in\uad$ satisfies $\bar u\in L^\infty(X)$ even in the case $p=2$ and $\umin = -\infty$ or $\umax = +\infty$.

Hereafter $\bar u$ denotes a local minimizer of \Pb. Associated with $\bar u$, we define the cone of critical directions
\[
C_{\bar u} = \Big\{v\in L^2(X):\left\{
\begin{array}{l} v(x)\geq 0\text{ if }\bar u(x)=\umin,\\ v(x)\leq 0\text{ if }\bar u(x)=\umax,\\ v(x) = 0\text{ if }\Phi(\bar u)(x) + \tikhonov\bar u(x) \neq 0
\end{array}
\right.\ \text{a.e. }[\mu]
\Big\}.
\]

\begin{assumption}\label{A2.6}
The local minimizer $\bar u\in\uad $ satisfies the second-order sufficient optimality condition $J''(\bar u)v^2>0\ \forall v\in C_{\bar u}\setminus\{0\}$ and the strict complementarity condition
\[
\mu\big(\{x\in X: \bar u(x)\in\{\umin,\umax\}\text{ and }\tikhonov \bar u(x) +\Phi(\bar u)(x)=0\}\big) =0.
\]
\end{assumption}
For every $\tau\geq 0$ we define the sets
\[
X_{\bar u}^{\taup} =  \{x\in X:\ \Phi(\bar u)(x) + \tikhonov\bar u(x) > \tau\},\
X_{\bar u}^{\taum} =  \{x\in X:\ \Phi(\bar u)(x) + \tikhonov\bar u(x) < -\tau\},
\]
$X^\tau_{\bar u}  =  X_{\bar u}^{\taup} \cup X_{\bar u}^{\taum}$, and the vector space
$E_{\bar u}^\tau = \{v\in L^2(X): v(x) = 0\text{ if }x\in X^{\tau}_{\bar u}\}$.
From \eqref{E2.7} we get that $\bar u(x)=\umin$ if $x\in X_{\bar u}^{\taup}$ and $\bar u(x)=\umax$ if $x\in X_{\bar u}^{\taum}$ a.e. $[\mu]$.

\begin{lemma}\label{L2.7}
Let $\rho_0 > 0$ be chosen such that \eqref{E2.1}, \eqref{E2.2} and the assumption \ref{A2.4} hold with $u = \bar u$ for the same radius $\rho_0$. Under the assumptions \ref{A2.1}, \ref{A2.2}, and \ref{A2.6}, there exist $\nu>0$,  $\tau>0$, and $\rho_{\textsc{ssc}}\in(0,\rho_0)$ such that
\begin{align}
&J''(\bar u) v^2\geq  \nu\|v\|_{L^2(X)}^2\ \forall v\in E_{\bar u}^\tau, \label{E2.8}\\
&J''( u) v^2\geq \frac{\nu}{2}\|v\|_{L^2(X)}^2\ \ \forall v\in E_{\bar u}^\tau\ \text{ and }\ \forall u\in B^p_{\rho_{\textsc{ssc}}}(\bar u). \label{E2.9}
\end{align}
\end{lemma}

\begin{proof}
Inequality \eqref{E2.8} is proved in \cite[Theorem 2.3]{Casas2024}. Let us prove \eqref{E2.9}. From Assumption \ref{A2.2} we deduce the existence of $\rho_{\textsc{ssc}}\in(0,\rho_0)$ such that
\[
\|(\Phi'(u)-\Phi'(w))v\|_{L^2(X)}\leq \frac{\nu}{2} \|v\|_{L^2(X)}\ \forall w\in B_{\rho_{\textsc{ssc}}}^p(u)\text{ and }\forall v\in L^2(X),
\]
whence
\begin{align*}
|(J''(\bar u)-J''(u))v^2|\leq & \int_X \left|[(\Phi'(u)-\Phi'(w))v]v\right|\dmu \\
\leq &\|(\Phi'(u)-\Phi'(w))v\|_{L^2(X)} \|v\|_{L^2(X)} \leq
\frac{\nu}{2} \|v\|_{L^2(X)}^2.
\end{align*}
This inequality and \eqref{E2.8} imply \eqref{E2.9}.
\end{proof}

\section{An SQP method to solve \Pb}
\label{ZZ-S3}
\setcounter{equation}{0}
Let us describe the SQP method as Newton's method for a generalized equation.
We define $F:\mathcal{A}\longrightarrow L^p(X)$ by $F(u) = \Phi(u) +\tikhonov u$.
 For any control $u\in L^2(X)$, we introduce the normal cone of $\uad$ at $u$ as
\[N(u) = \left\{\begin{array}{ll}
\{w\in L^2(X):\ \int_X w(v-u)d\mu\leq 0\ \forall v\in\uad\}&\text{ if }u\in\uad,\\
\emptyset &\text{ if }u\not\in\uad.
\end{array}
\right.\]
The first-order optimality condition \eqref{E2.6} can be written as the generalized equation $0\in F(\bar u)+ N(\bar u)$. The generalized Newton method is formulated as follows: for a given $u_0\in \uad$ and $n\geq 0$, $u_{n+1}$ is a solution of
\begin{equation}\label{ZZ-E3.10}
0\in F(u_n) + F'(u_n) (u_{n+1}-u_n) + N(u_{n+1}).
\end{equation}
Notice that \eqref{ZZ-E3.10} is the first-order optimality condition of the constrained quadratic problem
\[
(Q_n)\qquad \min_{u\in\uad}\frac{1}{2}J''(u_n)(u-u_n)^2+J'(u_n)u.
\]

In principle, the algorithm as written above need not be well defined, because the problem $(Q_n)$ may have no solution or more than one solution. Our objective is to prove that there exists $\rho>0$ such that starting the iterations at a control $u_0\in B^p_\rho(\bar u)$, each problem $(Q_n)$ has a unique {\em local} solution in $B^p_\rho(\bar u)$, and that the sequence built in this way converges quadratically to $\bar u$ in $L^p(X)$ and in $L^\infty(X)$. To prove this, we perform the following steps. In Section \ref{S3.1a} we modify the iteration \eqref{ZZ-E3.10} by replacing $\uad$ by a new set $\widehat{\uad}$, which includes the constraint $u(x) = \bar u(x)$ if $x\in X_{\bar u}^\tau$. For this modified iteration we prove  the existence of $\rho > 0$ such that, under the assumption $u_n\in B^p_\rho(\bar u)$, the iteration has a unique solution and the algorithm converges quadratically in $L^p(X)$ and in $L^\infty(X)$. Later in Section \ref{ZZ-S3.2} we show that the elements of the sequence obtained with the modified method are local solutions of $(Q_n)$, so we can obtain them without previous knowledge of $X^\tau_{\bar u}$. We formulate our main result about the convergence of Algorithm \ref{Alg1} in Theorem \ref{ZZ-T3.9}.

\LinesNumbered
\begin{algorithm2e}
\caption{SQP method to solve \Pb.}\label{Alg1}
\DontPrintSemicolon
Initialize. Choose $u_0\in B^p_\rho(\bar u)$. Set $n=0$.\;
\For{$n\geq 0$}{
Find a local solution of the constrained quadratic problem
\begin{align*}
\mathrm{(}Q_n\mathrm{)}\quad & \min_{v\in\uad-\{u_n\}} \frac{1}{2}\int_\Omega \big(\kappa v +\Phi'(u_n)v\big) v\dx + \int_\Omega (\kappa u_n+ \Phi(u_n) \big)v\dx
\end{align*} \;
Name $v_n$ the obtained solution.
\;
Set $u_{n+1}=u_n+v_n$ and $n=n+1$.\;
}
\end{algorithm2e}

\subsection{The modified method}\label{S3.1a}

Following \cite{Troltz1999}, we consider the set
\[\widehat{\uad}= \{u\in\uad:\ u(x) = \bar u(x)\text{ if }x\in X_{\bar u}^\tau\}.\]
For any control $u\in L^2(X)$, we will denote $\widehat{N}(u)$ the normal cone of $\widehat{\uad}$ at $u$. The modified generalized Newton step is given by
\begin{equation}\label{ZZ-E3.11}
0\in F(u_n) + F'(u_n) (u_{n+1}-u_n) + \widehat{N}(u_{n+1}).
\end{equation}
\begin{lemma}\label{le::existuniqhatQn}
Let $\rho_{\textsc{ssc}}$ be as introduced in Lemma 2.7. For every $u_n\in B^p_{\rho_{\textsc{ssc}}}(\bar u)$, equation \eqref{ZZ-E3.11} has a unique solution $u_{n+1}\in\widehat{\uad}$.
\end{lemma}
\begin{proof}
Equation \eqref{ZZ-E3.11} is the first-order optimality condition of the constrained quadratic problem
\[
(\widehat{Q}_n)\qquad \min_{u\in \widehat{\uad}}
\frac{1}{2}J''(u_n) (u-u_n)^2 + J'(u_n) u.
\]
It is sufficient to show that this problem is convex and has a unique solution. Let us denote $\mQ(u) = \frac12 J''(u_n) (u-u_n)^2 + J'(u_n) u$. We have that $\mQ''(u)v^2=J''(u_n)v^2$ for all $u\in\widehat{\uad}$ and all $v\in L^2(X)$. For all $u_a,u_b\in\widehat{\uad}$, $u_a-u_b\in E^\tau_{\bar u}$, so inequality \eqref{E2.9} implies that $\mQ''(u)(u_a-u_b)^2 \geq \frac{\nu}{2}\|u_a-u_b\|^2_{L^2(X)}$ and consequently $\mQ$ is strictly convex in $\widehat{\uad}$; see e.g. \cite[Th. 7.4-3]{Ciarlet1982}.
Since $\mathcal{Q}$ is convex and continuous, it is weakly lower semicontinuous. Noting that $\widehat{\uad}$ is convex and closed, we already have that $(\widehat{Q}_n)$ has a unique solution $u_{n+1}\in\widehat{\uad}$ if $\widehat{\uad}$ is bounded. In other case,
using \eqref{E2.3}, we can write
\begin{align*}
  \mQ(u) =&
  \frac{1}{2}J''(u_n)(u-\bar u+\bar u - u_n)^2+ J'(u_n)u  \\
   = &\frac{1}{2}J''(u_n)(u-\bar u)^2 + J''(u_n)(\bar u -u_n,u-\bar u) + \frac{1}{2}J''(u_n)(\bar u-u_n)^2+ J'(u_n)u
\end{align*}
On one hand, the term
$J''(u_n)(\bar u -u_n,u-\bar u) + \frac{1}{2}J''(u_n)(\bar u-u_n)^2+ J'(u_n)u$
is linear-affine in $u$. On the other hand, if $u\in\widehat{\uad}$ then $u-\bar u\in E^\tau_{\bar u}$, and \eqref{E2.9} implies that
\[\frac{1}{2}J''(u_n)(u-\bar u)^2\geq \frac{\nu}{2}\|u-\bar u\|_{L^2(X)}^2\]
and hence $\mQ$ is coercive. Therefore, $(\widehat{Q}_n)$ has a unique solution $u_{n+1}\in\widehat{\uad}$.
\end{proof}
We will prove the existence of $\rho > 0$ such that if $u_n \in B^p_\rho(\bar u)$ then the solution $u_{n+1}$ of \eqref{ZZ-E3.11} is also an element of the same ball. This implies that the modified Newton method is well defined, which is the first step towards the convergence. To do this, for every $u_n\in\uad$ and every $u\in L^2(X)$, we define
\begin{equation}\label{ZZ-E3.12}
  \delta_n(u) = F(\bar u)-F(u_n) + F'(\bar u)(u-\bar u)-F'(u_n)(u-u_n).
\end{equation}
A straightforward computation shows that \eqref{ZZ-E3.11} is equivalent to
\begin{equation}\label{ZZ-E3.13}
  \delta_n(u_{n+1})\in F(\bar u)+F'(\bar u)(u_{n+1}-\bar u) + \widehat{N}(u_{n+1}).
\end{equation}
Let us prove some auxiliary results about the following perturbed problem
\begin{equation}\label{ZZ-E3.14}
\delta\in F(\bar u)+F'(\bar u)(v-\bar u) + \widehat{N}(v).
\end{equation}

\begin{lemma}\label{ZZ-L3.2}
For any $\delta\in L^2(X)$, the problem \eqref{ZZ-E3.14}
has a unique solution $v_\delta\in\widehat{\uad}$. Moreover, $v_\delta$ satisfies
\begin{equation}\label{ZZ-E3.15}
  v_\delta(x) = \proj_{[\alpha,\beta]}\left[\frac{-1}{\kappa} \big(\Phi'(\bar u) (v_\delta-\bar u)+\Phi(\bar u)-\delta\big)(x)\right]\ \text{for a.a.}\,[\mu] \ x\in X\setminus X^\tau_{\bar u}.
\end{equation}
Further, if $\delta\in L^\infty(X)$, then $v_\delta\in L^\infty(X)$. Finally, if $\delta=0$ then $\bar u$ is the unique solution of \eqref{ZZ-E3.14}.
\end{lemma}
\begin{proof}
  The generalized equation \eqref{ZZ-E3.14} is the first-order optimality condition of the constrained quadratic problem
\begin{equation*}
(\widehat{Q}_\delta)\qquad \min_{v\in \widehat{\uad}}
\frac{1}{2}J''(\bar u) (v-\bar u)^2 + J'(\bar u) v-\int_X\delta vd\mu.
\end{equation*}
By definition of $\widehat{\uad}$, we have that for all $v \in \widehat{\uad}$ the identity $v(x)-\bar u(x) =0$ holds for a.a. $[\mu]$ $x\in X^\tau_{\bar u}$, and hence $v-\bar u\in E^\tau_{\bar u}$. From the inequality \eqref{E2.8}, we infer that $(\widehat{Q}_\delta)$ is a strictly convex problem and $J''(\bar u) (v-\bar u)^2$ is coercive on $\widehat{\uad}$. Hence,  $(\widehat{Q}_\delta)$ has a unique solution, which is characterized by \eqref{ZZ-E3.14}.

The projection formula \eqref{ZZ-E3.15} follows from a standard  argument.
The $L^\infty(X)$ regularity in the case of bounded constraints is a consequence of \eqref{ZZ-E3.15}. Suppose one of the box constraints is missing. From the assumption \ref{A2.1} we deduce that $\Phi'(\bar u) (v_\delta-\bar u)+\Phi(\bar u)\in L^\infty(X)$ and hence the $L^\infty(X)$ regularity of $\delta$ yields the same regularity for $v_\delta$.

Since $\bar u \in \widehat{\uad}\subseteq \uad$, we know that $N(\bar u)\subseteq \widehat{N}(\bar u)$, so $0\in F(\bar u)+N(\bar u)\subseteq F(\bar u)+\widehat{N}(\bar u)$ and the last statement follows immediately.
\end{proof}

Next we address the Lipschitz stability of the solution of $\widehat{Q}_\delta$ with respect to small perturbations. First, we introduce some notation. For $q\geq p$, we will denote
\[
n_{p,q} = \left\{\begin{array}{ll}
 \mu(X)^{\frac{q-p}{qp}}& \text{ if }p\leq q<\infty,\\
 \mu(X)^{\frac{1}{p}}& \text{ if }p<q=\infty,\\
 1 & \text{ if }p = q= \infty.
\end{array}
\right.
\]
Notice that for $q\geq p$ and any $\rho >0$, $B^q_{\rho/n_{p,q}}(u)\subset B^p_{\rho}(u)$. Related with assumptions \ref{A2.3} and \ref{A2.4}, we define the following constants $\hat c_i$, that will appear in a natural way along the proof:
\[
\hat c_0 = \displaystyle\frac{n_{2,\infty}}{\nu} \text{ and } \hat c_i =\frac{1}{\tikhonov}\left(  \|\Phi'(\bar u)\|_{\mathcal{L}(L^{p_{i-1}}(\mu),L^{p_i}(X))} \hat c_{i-1}+n_{p_i,\infty}\right)\ i = 1, \ldots, N+1.
\]
\begin{theorem}\label{ZZ-T3.3}Given $\delta_a,\delta_b\in L^\infty(X)$,  let $v_a,v_b\in\widehat{\uad}$ be the solutions of $(\widehat{Q}_{\delta_a})$ and $(\widehat{Q}_{\delta_b})$, respectively. Then, there exists a constant $\hat c>0$ such that
\begin{equation}\label{ZZ-E3.16}\|v_a-v_b\|_{L^{q}(X)} \leq \hat c \|\delta_a-\delta_b\|_{L^\infty(X)}\ \text{ for all } q\in[1,\infty]
\end{equation}
\end{theorem}
\begin{proof}Obviously, it is enough to prove it for $q=\infty$.
  We notice that for any $\delta\in L^2(X)$, with the change $w=v-\bar u$, problem $(\widehat{Q}_\delta)$ is equivalent to
\[(\widehat{Q}_\delta')\qquad \min_{w\in \widehat{\uad}-\{\bar u\}} \frac{1}{2}J''(\bar u) w^2+ J'(\bar u)w -\int_X\delta w d\mu
\]
Taking into account the second-order condition \eqref{E2.8},  the Lipschtz stability property in $L^2(X)$ proved in Lemma \ref{L2.16}, see the appendix, yields
\begin{equation}\label{ZZ-E3.17}
  \|v_a-v_b\|_{L^2(X)}\leq  \frac{1}{\nu}\|\delta_a-\delta_b\|_{L^2(X)} \le \hat{c}_0\|\delta_a-\delta_b\|_{L^\infty(X)}.
\end{equation}
Now we use a bootstrapping argument. Subtracting the respective projection formulas for $v_b$ and $v_a$, see \eqref{ZZ-E3.15}, and using the Lipschtiz property of the projection, we obtain
\[|v_a(x)-v_b(x)|\leq \frac{1}{\tikhonov}\big|\left(-\Phi'(\bar u)(v_a-v_b) + \delta_a-\delta_b\right)(x)\big| \text{ for a.a. }[\mu]\ x\in X\setminus X^\tau_{\bar u}.\]
Starting with \eqref{ZZ-E3.17} and using Assumption \ref{A2.3} we get for $j=1,\ldots,N+1$
\begin{align*}
 \| v_a&-v_b\|_{L^{p_j}(X)} =   \|v_a-v_b\|_{L^{p_j}(X\setminus X^\tau_{\bar u})} \\
  \leq &\frac{1}{\tikhonov}\| -\Phi'(\bar u)(v_a-v_b) + \delta_a-\delta_b \|_{L^{p_j}(X\setminus X^\tau_{\bar u})} \leq \frac{1}{\tikhonov}\| -\Phi'(\bar u)(v_a-v_b) + \delta_a-\delta_b \|_{L^{p_j}(X)} \\
  \leq & \frac{1}{\tikhonov}\|\Phi'(\bar u)\|_{\mathcal{L}(L^{p_{j-1}}(X),L^{p_j}(X))}  \|v_a-v_b\|_{L^{p_{j-1}}(X)} + \frac{n_{p_j,\infty}}{\tikhonov}\| \delta_a-\delta_b \|_{L^\infty(X)} \\
  \leq & \frac{1}{\tikhonov}\|\Phi'(\bar u)\|_{\mathcal{L}(L^{p_{j-1}}(X),L^{p_j}(X))} \hat c_{j-1} \| \delta_a-\delta_b \|_{L^\infty(X)} + \frac{n_{p_j,\infty}}{\tikhonov}\| \delta_a-\delta_b \|_{L^\infty(X)}\\
  \leq &  \frac{1}{\tikhonov}\big( \|\Phi'(\bar u)\|_{\mathcal{L}(L^{p_{j-1}}(X),L^{p_j}(X))} \hat c_{j-1}+ n_{p_j,\infty}\big)  \| \delta_a-\delta_b \|_{L^\infty(X)} =  \hat c_j\| \delta_a-\delta_b \|_{L^\infty(X)}.
\end{align*}
and the proof is complete taking $\hat c = \max\{1,\mu(X)\} \hat c_{N+1}$.
\end{proof}

We prove now some properties of the perturbations.
\begin{lemma}\label{ZZ-L3.4}
The following properties hold:
\begin{itemize}
\item[\textup{(i)}] For every $u_n\in \mA$ and every $u\in L^p(X)$, we have that $\delta_n(u)\in L^\infty(X)$.
\item[\textup{(ii)}] Let $\rho_0$ be the positive number introduced in Lemma \ref{L2.7}. Let $u_n$ be an element of $B^p_{\rho_0}(\bar u)$ and $u_a,u_b\in L^{q}(X)$ for 
    $q\in[p,\infty]$. Then, it fulfills
\[\|\delta_n(u_a)-\delta_n(u_b)\|_{L^\infty(X)} \leq n_{p,q} L_{\bar u,\Phi'} \|\bar u- u_n\|_{L^p(X)} \|u_a-u_b\|_{L^{q}(X)}.\]
\item[\textup{(iii)}] Suppose that $u_n,u\in B^p_\rho(\bar u)$ for some $\rho \in (0,\rho_0]$, then we have the estimate $\|\delta_n(u)\|_{L^\infty(X)} \leq (L_{\bar u,\Phi} + 3 M_{\bar u,\Phi'})\rho$.
\item[\textup{(iv)}]
Let 
$q\in[p,\infty]$ and suppose that  $u_n\in B^{q}_{\rho_0/n_{p,q}}(\bar u)$. Then, it holds:
\[\|\delta_n( \bar u)\|_{L^\infty(X)}\leq \frac{1}{2}
n_{p,q}^2 L_{\bar u,\Phi'}  \|\bar u-u_n\|_{L^{q}(X)}^2.\]
\end{itemize}
\end{lemma}
\begin{proof}
\begin{itemize}
\item[\textup{(i)}]
From the definition of $\delta_n(u)$ and using Assumption \ref{A2.1} together with the fact that $\bar u\in L^\infty(X)$, we have that
  \begin{align*}
    \delta_n(u) = &\, \Phi(\bar u)+\tikhonov\bar u-  \Phi( u_n)-\tikhonov u_n\\
    & +\Phi'(\bar u) (u-\bar u)+\tikhonov (u-\bar u)-\Phi'(u_n)(u-u_n)-\tikhonov(u - u_n) \\
    = &\, \Phi(\bar u)-  \Phi( u_n)+\Phi'(\bar u) (u-\bar u)-\Phi'(u_n)(u-u_n)\in L^\infty(X).
  \end{align*}
\item[\textup{(ii)}]
From (i), we have that both $\delta_n(u_a)$ and $\delta_n(u_b)$ are elements of $L^\infty(X)$. Using the definition of $\delta_n(u)$ and Assumption \ref{A2.4}, we obtain

\begin{align*}
\|\delta_n(u_a)-\delta_n(u_b)\|_{L^\infty(X)} &=  \|(\Phi'(\bar u)-\Phi'( u_n))(u_a-u_b)\|_{L^\infty(X)} \\
&\leq L_{\bar u,\Phi'} \|\bar u- u_n\|_{L^p(X)} \|u_a-u_b\|_{L^p(X)}\\
&\leq n_{p,q}L_{\bar u,\Phi'} \|\bar u- u_n\|_{L^p(X)} \|u_a-u_b\|_{L^{q}(X)}.
\end{align*}
\item[\textup{(iii)}]
 Using \eqref{E2.1}, \eqref{E2.2}, and $u_n,u\in B^p_\rho(\bar u)$, we have
\begin{align*}
&\|\delta_n(u)\|_{L^\infty(X)}  \leq \|\Phi(\bar u)-  \Phi( u_n)\|_{L^\infty(X)} + \| \Phi'(\bar u) (u-\bar u)\|_{L^\infty(X)}\\
&+\|\Phi'(u_n)(u-u_n)\|_{L^\infty(X)} \leq L_{\bar u,\Phi}\|\bar u-u_n\|_{L^p(X)} +
       M_{\bar u,\Phi'} \|u-\bar u\|_{L^p(X)} \\
&+ M_{\bar u,\Phi'} \|u-\bar u\|_{L^p(X)} + M_{\bar u,\Phi'} \|u_n-\bar u\|_{L^p(X)}
       \leq  (L_{\bar u,\Phi} + 3 M_{\bar u,\Phi'})\rho,
  \end{align*}
\item[\textup{(iv)}]
We have that
\begin{align*}
\delta_{n}(\bar u) = & F(\bar u) - F(u_n) - F'(u_n)(\bar u - u_n)
\\
  = & \Phi(\bar u) +\tikhonov \bar u - \Phi(u_n) -\tikhonov u_n - \Phi'(u_n)(\bar u-u_n) -\tikhonov(\bar u-u_n) \\
  = & \int_0^1 (\Phi'(u_n + \theta(\bar u - u_n)) - \Phi'(u_n)) (\bar u - u_n)\mathrm{d}\theta.
\end{align*}
Since $B^{q}_{\rho_0/n_{p,q}}(\bar u)\subset B^p_{\rho_0}(\bar u)$, from Assumption \ref{A2.4} we readily deduce that
\begin{align}
\|(\Phi'(u_n + &\theta(\bar u - u_n))-\Phi'(u_n)) (\bar u-u_n)\|_{L^\infty(X)}\notag\\
\leq &
 L_{\bar u,\Phi'}  \|\theta(\bar u-u_n)\|_{L^{p}(X)} \|\bar u-u_n\|_{L^p(X)}
 \leq
 \theta n_{p,q}^2L_{\bar u,\Phi'}   \|\bar u-u_n\|_{L^{q}(X)}^2.\notag
\end{align}
The proof concludes by using the identity $\int_0^1\theta d\theta=1/2$.
\end{itemize}
\end{proof}

For $q\in[p,\infty]$ we set
\begin{equation}\label{ZZ-E3.18}
 \rho_{q}^\star = \min\{\frac{\rho_{\textsc{ssc}}}{n_{p,q}},\displaystyle \frac{1}{2 n^2_{p,q}\hat c L_{\bar u,\Phi'} }\}
\end{equation}
where $L_{\bar u,\Phi'}$, $\rho_{\textsc{ssc}}$   and $\hat c$ were introduced respectively in Assumption \ref{A2.4}, Lemma \ref{L2.7}, and Theorem \ref{ZZ-T3.3}.
\begin{lemma}\label{ZZ-L3.5}
Let 
$q\in[p,\infty]$ and let $\rho>0$ be such that $\rho \leq \rho_{q}^\star$.
Assume $u_0\in B^{q}_{\rho}(\bar u)$. Then, the sequence spanned by \eqref{ZZ-E3.11} is well defined, $u_n\in  B^{q}_{\rho}(\bar u)$ for all $n\geq 0$, $u_n\to\bar u$ in $L^{q}(X)$ and
\begin{equation}\label{ZZ-E3.19}
\|u_{n+1}-\bar u\|_{L^{q}(X)} \leq n_{p,q}^2 \hat c L_{\bar u,\Phi'} \|u_n-\bar u\|_{L^{q}(X)}^2\ \forall n\geq 0.
\end{equation}
\end{lemma}
\begin{proof}
We prove by induction that $u_n\in  B^{q}_{\rho}(\bar u)$ for all $n\geq 0$. By assumption, the statement is true for $n=0$. Suppose it is true for $n$ and let us prove it for $n+1$.

By Lemma \ref{le::existuniqhatQn}, and noting that $B^{q}_{\rho}(\bar u)\subset B^p_{\rho_{\textsc{ssc}}}(\bar u)$,  \eqref{ZZ-E3.11} has a unique solution $u_{n+1}\in\widehat{\uad}$. Furthermore, by \eqref{ZZ-E3.13}, we now that  $u_{n+1}$ solves $(\widehat{Q}_\delta)$ for $\delta = \delta_n(u_{n+1})$.
Let $v\in L^\infty(X)$ be the unique solution of $(\widehat{Q}_\delta)$ for $\delta = \delta_n(\bar u)$. From Theorem \ref{ZZ-T3.3}, Lemma \ref{ZZ-L3.4}(ii), and the condition on $\rho$ we have that
\begin{align*}
  \|u_{n+1}-v\|_{L^{q}(X)}\leq & \hat c\|\delta_{n}(u_{n+1})-\delta_n(\bar u)\|_{L^\infty(X)} \\
  \leq &n_{p,q} \hat c  L_{\bar u,\Phi'} \|\bar u-u_n\|_{L^p(X)} \|u_{n+1}-\bar u\|_{L^{q}(X)} \\
  \leq &n_{p,q}^2 \hat c  L_{\bar u,\Phi'} \rho \|u_{n+1}-\bar u\|_{L^{q}(X)} \leq \frac{1}{2}\|u_{n+1}-\bar u\|_{L^{q}(X)},
\end{align*}
where we have used the induction hypothesis to get $\|\bar u-u_n\|_{L^{p}(X)}\leq n_{p,q}\|\bar u-u_n\|_{L^{q}(X)}\leq n_{p,q} \rho$.
Therefore
\begin{align*}
\|u_{n+1}-\bar u\|_{L^{q}(X)}\leq &\|u_{n+1}-v\|_{L^{q}(X)} +\|v-\bar u\|_{L^{q}(X)}\\
\leq & \frac{1}{2}\|u_{n+1}-\bar u\|_{L^{q}(X)} +\|v-\bar u\|_{L^{q}(X)}.
\end{align*}
This inequality, Theorem \ref{ZZ-T3.3}, the fact that $\bar u$ is the solution of $(\widehat{Q}_0)$, and the estimate in Lemma \ref{ZZ-L3.4}(iv) yield
\begin{align}
\|u_{n+1}-\bar u\|_{L^{q}(X)} \leq & 2\|v-\bar u\|_{L^{q}(X)}
\leq  2 \hat c\|\delta_n(\bar u)\|_{L^\infty(X)} \notag\\
\leq & n_{p,q}^2 \hat c L_{\bar u,\Phi'} \|\bar u-u_n\|_{L^{q}(X)}^2.\label{ZZ-E3.20}
  \end{align}
By induction hypotheses, $\|\bar u-u_n\|_{L^{q}(X)}\leq \rho$. Using this fact, \eqref{ZZ-E3.20} and the condition imposed on $\rho$ we obtain
  \begin{align}\label{ZZ-E3.21}
  \|u_{n+1}-\bar u\|_{L^{q}(X)} \leq & n_{p,q}^2\hat c L_{\bar u,\Phi'} \rho^2\leq \frac{1}{2}\rho.
  \end{align}
So, $u_{n+1}\in B^{q}_{\rho}(\bar u)$ and the sequence is well defined. Estimate \eqref{ZZ-E3.19} was proved in \eqref{ZZ-E3.20}. This estimate leads to
  \[\|u_{n+1}-\bar u\|_{L^{q}(X)}\leq \frac{1}{2^{2^n}}\frac{1}{n_{p,q}^2\hat c L_{\bar u,\Phi'}}, \]
  and hence $u_n\to\bar u$ in $L^{q}(X)$.
\end{proof}
So far, we have quadratic convergence in $L^{q}(X)$ provided $u_0\in B^{q}_{\rho_q^\star}(\bar u)$ for 
$q\in[p,\infty]$. We show now that for any $q = (p,\infty]$ it is enough that $u_0\in B^p_\rho(\bar u)$ for an appropriate $\rho$.

\begin{corollary}\label{ZZ-C3.6}Consider $q\in (p,\infty]$ and suppose $u_0\in B^p_\rho(\bar{u})$ for some $\rho \in (0,\rho_p^\star]$ satisfying
\begin{equation}\label{ZZ-E3.22} \hat c (L_{\bar u,\Phi} + 3 M_{\bar u,\Phi'})\rho \leq \rho_{q}^\star \end{equation}
Then, the sequence spanned by \eqref{ZZ-E3.11} is well defined, converges to $\bar u$ in $L^\infty(X)$, and
\[
\|u_{n+1}-\bar u\|_{L^{q}(X)} \leq n_{p,q}^2 \hat c L_{\bar u,\Phi'} \|u_n-\bar u\|_{L^{q}(X)}^2\ \forall n\geq 1.
\]
\end{corollary}
\begin{proof}
Using again that  $u_{1}$ solves $(\widehat{Q}_\delta)$ for $\delta = \delta_0(u_{1})$ and $\bar u$ is the solution of $(\widehat{Q}_0)$, from Theorem \ref{ZZ-T3.3}, and Lemma \ref{ZZ-L3.4}(iii), we obtain that
  \[\|u_{1}-\bar u\|_{L^{q}(X)}\leq \hat c \|\delta_0(u_{1})\|_{L^\infty(X)} \leq \hat c (L_{\bar u,\Phi} + 3 M_{\bar u,\Phi'})\rho.\]
This means that $u_1\in B^{q}_{\rho_{q}}(\bar u)$ for $\rho_{q} = \hat c (L_{\bar u,\Phi} + 3 M_{\bar u,\Phi'})\rho$. Condition \eqref{ZZ-E3.22} implies that $\rho_{q}\leq \rho_{q}^\star$, and we can apply Lemma \ref{ZZ-L3.5} for quadratic convergence in $L^{q}(X)$ for the shifted sequence $u_n' = u_{n+1}$ for $n\geq 0$.
\end{proof}
\subsection{Back to the original method}\label{ZZ-S3.2}
Now we study the relation of the solution of \eqref{ZZ-E3.11} with solutions of \eqref{ZZ-E3.10}.
\begin{lemma}\label{ZZ-L3.7}
Let $v\in\widehat{\uad}$ be the unique solution of problem 
$(\widehat{Q}_\delta)$
 for some $\delta$ such that $(1+n_{p,\infty}\hat c M_{\bar u,\Phi'})\|\delta\|_{L^\infty(X)}\leq \tau$. Then $v$ solves
\[\delta\in F(\bar u)+F'(\bar u)(v-\bar u) + N(v).\]
\end{lemma}
\begin{proof}
  Consider $u\in\uad$ and define
  \[\hat u = \left\{\begin{array}{cc}
                      \bar u & \text{ in }X^\tau_{\bar u} \\
                      u & \text{ in }X\setminus X^\tau_{\bar u}
                    \end{array}
  \right.
  \]
  We have that
  \begin{align*}
    \int_X & (\Phi'(\bar u)(v-\bar u) + \tikhonov v +\Phi(\bar u) -\delta)(u-v)d\mu  \\
    = & \int_{X^\tau_{\bar u}}  (\Phi'(\bar u)(v-\bar u) + \tikhonov v +\Phi(\bar u) -\delta)(u-v)d\mu
    \\
    &
     + \int_{X}  (\Phi'(\bar u)(v-\bar u) + \tikhonov v +\Phi(\bar u) -\delta)(\hat u-v)d\mu.
  \end{align*}
  The last integral is nonnegative since $v$ solves \eqref{ZZ-E3.14}.
For the second integral, we first notice that, using that $\bar u$ is the unique solution of $(\widehat{Q}_0)$ and Theorem \ref{ZZ-T3.3} with $i = N+1$, we have that $|v(x)-\bar u(x)|\leq \hat c \|\delta\|_{L^\infty(X)}$ for a.a.$[\mu]$ $x\in X$. From \eqref{E2.2} we infer that
  \[\|\Phi'(\bar u)(v-u)\|_{L^\infty(X)} \leq n_{p,\infty}M_{\bar u,\Phi'} \|v-\bar u\|_{L^\infty(X)} \leq n_{p,\infty} \hat c M_{\bar u,\Phi'} \|\delta\|_{L^\infty(X)}.\]
 This yields
  \begin{equation}\label{ZZ-E3.23}|(\Phi'(\bar u)(v-u)-\delta)(x)|\leq (1+n_{p,\infty}\hat c M_{\bar u,\Phi'})\|\delta\|_{L^\infty(X)}\leq \tau\ \text{ a.e.}[\mu] \text{ in } X.
  \end{equation}
  Using that $v=\bar u$ in $X^\tau_{\bar u}$ and the definition of $X^{\taup}_{\bar u}$, we have that
  \begin{equation}\label{ZZ-E3.24}\bar\Phi(\bar u)(x) + \tikhonov v(x) = \bar\Phi(\bar u)(x) + \tikhonov \bar u(x) >\tau\text{ for a.e.}[\mu]\ x\in X^{\taup}_{\bar u}.
  \end{equation}
  In a similar way we have that
  \begin{equation}\label{ZZ-E3.25}\bar\Phi(\bar u)(x) + \tikhonov v(x) = \bar\Phi(\bar u)(x) + \tikhonov \bar u(x) <-\tau\text{ for a.e.}[\mu]\ x\in X^{\taum}_{\bar u}.
  \end{equation}
Then we can write
  \begin{align*}
   \int_{X^\tau_{\bar u}} & (\Phi'(\bar u)(v-\bar u) +\tikhonov v +\Phi(\bar u) -\delta)(u-v)d\mu \\
    =  &  \int_{X^{\taup}_{\bar u}}(\Phi'(\bar u)(v-\bar u)+ \tikhonov v+\Phi(\bar u)-\delta)(u-\alpha)d\mu \\
    & + \int_{X^{\taum}_{\bar u}}(\Phi'(\bar u)(v-\bar u) +\tikhonov v+\Phi(\bar u)-\delta)(u-\beta)d\mu.
  \end{align*}
  From \eqref{ZZ-E3.23} and \eqref{ZZ-E3.24}, and noting that $u\geq \alpha$, we have that the integrand in the first integral is the product of two nonnegative numbers.  On the other hand, \eqref{ZZ-E3.23} and \eqref{ZZ-E3.25}, together with the fact that $u\leq \beta$ imply that the integrand in the second integral is the product of two non positive numbers, so the integral is nonnegative and the proof is complete.
\end{proof}
\begin{lemma}\label{ZZ-L3.8}Suppose $u_n\in  B^p_\rho(\bar u)$ for some $\rho\in(0,\rho_p^\star]$ satisfying
\begin{equation}\label{ZZ-E3.26}(1+n_{p,\infty}\hat c M_{\bar u,\Phi'})(L_{\bar u,\Phi} + 3 M_{\bar u,\Phi'})\rho \leq\tau,\end{equation}
and let $u_{n+1}\in\widehat{\uad}$ be the unique solution of \eqref{ZZ-E3.11}. Then, $u_{n+1}$ solves \eqref{ZZ-E3.10}, it is the unique solution of \eqref{ZZ-E3.10} in $B^p_\rho(\bar u)\cap\uad$, and it is the unique local solution of $(Q_n)$ in $B^p_\rho(\bar u)$.
\end{lemma}
\begin{proof}
Using the condition $\rho\leq \rho_p^\star$ we first deduce from Lemma \ref{le::existuniqhatQn} that \eqref{ZZ-E3.11} has a unique solution $u_{n+1}\in\widehat{\uad}$. From Lemma \ref{ZZ-L3.5}, we know that $u_{n+1}\in B^p_\rho(\bar u)$, and hence Lemma \ref{ZZ-L3.4}(iii) and \eqref{ZZ-E3.26} imply that
\[(1+n_{p,\infty}\hat c M_{\bar u,\Phi'}) \|\delta_n(u_{n+1})\|_{L^\infty(X)}\leq (1+n_{p,\infty}\hat c M_{\bar u,\Phi'}) (L_{\bar u,\Phi} +3 M_{\bar u,\Phi'})\rho\leq \tau.\]
In Lemma \ref{ZZ-L3.2} it was proved that for any $\delta\in L^2(X)$, $(\widehat{Q}_\delta)$ is a strictly convex problem and that its unique solution is characterized by equation \eqref{ZZ-E3.14}. We have that $u_{n+1}$ is the unique solution of \eqref{ZZ-E3.11} and that this equation is equivalent to \eqref{ZZ-E3.13} which is just \eqref{ZZ-E3.14} for the case $\delta = \delta_n(u_{n+1})$. Therefore $u_{n+1}$ is the unique solution of $(\widehat{Q}_{\delta_n(u_{n+1})})$. So we can apply Lemma \ref{ZZ-L3.7}, to obtain that
\[\delta_n(u_{n+1})\in F(\bar u)+F'(\bar u)(u_{n+1}-\bar u) + N(u_{n+1}),\]
which is equivalent to \eqref{ZZ-E3.10}.

Suppose now that $u\in B^p_\rho(\bar u)\cap\uad$ solves \eqref{ZZ-E3.10}.
For every $x\in X$ we have
\[  u(x) = \proj_{[\alpha,\beta]}\left[\frac{-1}{\tikhonov} \big(\Phi'(\bar u) (u-\bar u)+\Phi(\bar u)-\delta_n(u)\big)(x)\right]
\]
For $x\in X^{\taup}_{\bar u}$, $\Phi(\bar u)(x)+\tikhonov\alpha > \tau$.
Using this, the definition of $\delta_n(u)$, \eqref{E2.1} and \eqref{E2.2}, that both $u_n$ and $u$ are in $B^p_\rho(\bar u)$ and condition \eqref{ZZ-E3.26} , we obtain
\begin{align*}
  -\big(\Phi'(\bar u) & (u-\bar u)+\Phi(\bar u)-\delta_n(u)\big)(x)\\
  & =  -\Phi'(u_n)(u-u_n)(x)+\Phi(\bar u)(x) -\Phi(u_n)(x)-\Phi(\bar u)(x) \\
  & \leq  \|\Phi'(u_n)(u-u_n)\|_{L^\infty(X)} + \|\Phi(\bar u) -\Phi(u_n)\|_{L^\infty(X)} -\Phi(\bar u)(x) \\
  & < (2M_{\bar u,\Phi'}+L_{\bar u,\Phi})\rho +\tikhonov\alpha-\tau\leq \tau +\tikhonov\alpha-\tau = \tikhonov\alpha.
\end{align*}
So, from the projection formula, $u(x)=\alpha$. With a similar argument, we obtain that $u(x)=\beta$ if $x\in X^{\taum}_{\bar u}$. Therefore, $u\in\widehat{\uad}$ and consequently it solves \eqref{ZZ-E3.11}. Since the unique solution of \eqref{ZZ-E3.11} is $u_{n+1}$, then $u=u_{n+1}$.

Now, we prove that $u_{n+1}$  is a local solution of $(Q_n)$ in $B^p_\rho(\bar u)$. We collect for further reference that, hidden in the second and fourth lines of the last chain of inequalities, we have just proven that
\begin{equation}\label{ZZ-E3.27}
  \Phi'(u_n)(u_{n+1}-u_n)(x)+\Phi(u_n)(x) > -\tikhonov\alpha\text{ for a.a.}[\mu]\ x\in X^{\taup}_{\bar u}.
\end{equation}
As in the proof of Lemma \ref{le::existuniqhatQn}, we denote $\mQ(u)= \frac{1}{2}J''(u_n) (u-u_n)^2 + J'(u_n) u.$
The cone of critical directions for the problem $(Q_n)$ at $u_{n+1}$ is given by
\begin{align*}C^n_{u_{n+1}}&=\{v\in L^2(X):
  v(x)\geq 0 \text{ if }u_{n+1}=\alpha,\
  v(x)\leq 0  \text{ if }u_{n+1}=\beta, \\
&  v(x) = 0  \text{ if }(\Phi'(u_n) + \tikhonov\mathbb{I}) (u_{n+1}-u_n)(x) +
  \Phi(u_n)(x) + \tikhonov u_n(x)\neq 0.
\text{ a.e. }[\mu]\},
\end{align*}
where $\mathbb{I}:L^2(X) \longrightarrow L^2(X)$ denotes the identity operator. Let us show that $u_{n+1}$ satisfies the second-order condition $\mQ''(u_{n+1})v^2>0$ for all $v\in C^n_{u_{n+1}}\setminus\{0\}$. This is a sufficient condition for strict local optimality in the sense of $L^2(X)$; see, for instance,\cite[Theorem 2.3]{Casas-Troltzsch2012}.
First we note that $\mQ''(u_{n+1})v^2 = J''(u_n)v^2$ for all $v\in L^2(X)$. Now, using that $\rho\leq \rho_{\textsc{ssc}}$ we deduce from Lemma \ref{L2.7}(ii) that $J''(u_n)v^2\geq \frac{\nu}{2}\|v\|^2$ for all $v\in E^\tau_{\bar u}$. So the claim will follow if we prove that $C^n_{u_{n+1}}\subset E^\tau_{\bar u}$. Let $v$ be an element of $C^n_{u_{n+1}}$ and consider $x\in X^{\taup}_{\bar u}$. Using \eqref{ZZ-E3.27} and the fact that $u_{n+1}\geq \alpha$, we  deduce that
\begin{align*}
   (\Phi'(u_n) + \tikhonov\mathbb{I})& (u_{n+1}-u_n)(x) +
  \Phi(u_n)(x) + \tikhonov u_n(x)  \\
  = & \Phi'(u_n)(u_{n+1}-u_n)(x) +
  \Phi(u_n)(x) + \tikhonov u_{n+1}(x) \\
   >
  & -\tikhonov\alpha + \tikhonov u_{n+1}(x) \geq -\tikhonov\alpha +\tikhonov\alpha =0,
\end{align*}
and hence $v(x)=0$. A symmetric argument works for $x\in X^{\taum}_{\bar u}$. Therefore $v\in E^\tau_{\bar u}$ and $u_{n+1}$  is a local solution of $(Q_n)$.

Furthermore, if $u\in\uad\cap B_\rho^p(\bar u)$ is a local solution of $(Q_n)$, then it satisfies the first order optimality condition for $(Q_n)$, equation \eqref{ZZ-E3.10}. We have already proved that $u_{n+1}$ is the unique solution of \eqref{ZZ-E3.10} in $B_\rho^p(\bar u)$ and hence $u=u_{n+1}$.
\end{proof}
We finish with the main result of the paper.
\begin{theorem}\label{ZZ-T3.9}Let $\bar u\in\uad$ be a local solution of \Pb satisfying the second-order optimality condition $J''(\bar u)v^2 >0$ $\forall v \in C_{\bar u} \setminus\{0\}$ and the strict complementarity condition stated in Assumption \ref{A2.6}. Then, there exists $\bar\rho > 0$ such that for every $0 < \rho \le \bar\rho$ and $u_0 \in B_\rho^p(\bar u)$ we have: for every $n\geq 0$, $(Q_n)$ has a unique local solution $u_{n+1}$ in $B^p_\rho(\bar u)$, the sequence $\{u_n\}_{n = 0}^\infty$ generated by the algorithm \ref{Alg1} converges to $\bar u$ in $L^\infty(X)$, and for all $q\in[p,\infty]$,
\[
\|u_{n+1}-\bar u\|_{L^{q}(X)} \leq n_{p,q}^2 \hat c L_{\bar u,\Phi'} \|u_n-\bar u\|_{L^{q}(X)}^2\ \forall n\geq 1.
\]
\end{theorem}
\begin{proof}
Let us take $\bar \rho\in(0,\rho_p^\star]$ satisfying \eqref{ZZ-E3.22} and \eqref{ZZ-E3.26}, and an arbitrary number $\rho \in (0,\bar\rho]$.  Given $u_n\in B^p_\rho(\bar u)$, by Lemma \ref{ZZ-L3.8}, problem $(Q_n)$ has a unique local solution $u_{n+1}$ in $B^p_\rho(\bar u)$. Thus the sequence is well defined. Again by Lemma \ref{ZZ-L3.8} we infer that $u_{n+1}\in\widehat{\uad}$ and solves \eqref{ZZ-E3.11}. So we can apply Lemma \ref{ZZ-L3.5} and Corollary \ref{ZZ-C3.6} to obtain the quadratic convergence in 
$L^q(X)$.
\end{proof}
\begin{remark}
In the case $p>2$, since $-\infty<\umin<\umax<\infty$, for every $\rho >0$ we have that $B^2_r(\bar u)\cap\uad \subset B^p_{\rho}(\bar u)\cap\uad$
with $r = (\rho/\ell)^\frac{p}{2}$, where $\ell = \vert \umax-\umin\vert ^{\frac{p-2}{p}}$. Therefore, the method can be initialized with an initial point in an $L^2(X)$ neighborhood of $\bar u$ also in this case.
\end{remark}

\subsection{Extension to several controls}\label{ZZ-S3.3}
In a similar way, the technique can be extended to the following framework. For $n\geq 1$ and $i=1,\ldots,n$, consider $(X_i,\mathcal{S}_i,\mu_i)$ measure spaces with $\mu(X_i)<\infty$ and numbers $p^i\in[2,\infty]$, $\tikhonov_i>0$ and $-\infty\leq \alpha_i<\beta_i\leq\infty$. We require $-\infty< \alpha_i<\beta_i<\infty$ if $p^i>2$. Denote $B=\prod_{i=1}^n L^{p^i}(X_i)$ and $B_\infty = \prod_{i=1}^n L^{\infty}(X_i)$. We assume the existence of an open set $\mA\subset B$ and a function $\mJ:\mA\to\mathbb{R}$ of class $C^2$. We also assume the existence of a $C^1$ function $\bm\Phi:\mA\to B_\infty$ satisfying
\[
\mJ'(\bm u)\bm v=\sum_{i=1}^n\int_{X_i}\Phi_i(\bm u) v_i\dmu_i\ \forall\bm u\in\mA\text{ and }\forall \bm v\in B
\]
with the linear mappings $\Phi_i'(\bm u): B\to L^\infty(X_i)$ satisfying assumptions analogous to \ref{A2.2}, \ref{A2.3} and \ref{A2.4} for $i=1,\ldots,n$.

Consider the optimization problem
\[
\Pb\qquad\min_{\bm u\in\bm U_{\mathrm{ad}}} J(\bm u) = \mJ(\bm u) + \sum_{i=1}^n\frac{\tikhonov}{2}\|u_i\|_{L^2(X_i)}^2,
\]
where $\bm U_{\mathrm{ad}} = \{\bm u\in B:\ \alpha_i\leq u_i\leq\beta_i\text{ a.e. }[\mu_i]\ \forall i=1,\ldots,n\} \subset \mA$. Then it is straightforward to check that Theorem \ref{ZZ-T3.9} applies to this framework.

\section{Application to some optimal control problems}\label{ZZ-S4}
\setcounter{equation}{0}

In this section, we give some examples of optimal control problems that fit into the previous abstract framework. For these examples Theorem \ref{ZZ-T3.9} holds, which implies the quadratic convergence of the SQP method under the no gap second-order sufficient optimality condition and the strict complementarity assumption.

Along this section $\Omega$ denotes an open, bounded, connected subset of $\mathbb{R}^\dimension$, for $1\leq \dimension \leq 3$, with a boundary $\Gamma$ and $\omega\subseteq\Omega$ is a Lebesgue measurable set with positive measure. We will suppose that $\Gamma$ is Lipschitz if $2\leq\dimension\leq 3$. Some extra regularity of $\Gamma$ will be required in some cases. We also denote $Q = \Omega \times (0,T)$, $Q_\omega = \omega\times(0,T)$, and $\Sigma = \Gamma \times (0,T)$ for a given final time $T > 0$. The operator $A$ defined in $\Omega$ and the conormal derivative $\partial_{n_A}$ on $\Gamma$ are given by the expressions
\[
Ay = - \sum_{i,j=1}^{\dimension} \partial_{x_j} [ a_{ij}(x) \partial_{x_i} y] + a_0 y \quad \text{ and } \quad  \partial_{n_A}y = \sum_{i,j = 1}^da_{ij}(x)\partial_{x_i}y(x)n_j(x),
\]
where $n(x)$ denotes the outward unit normal vector to $\Gamma$ at the point $x$, $a_0\geq0$, and $a_0,a_{ij} \in L^{\infty}(\Omega)$ for $1\leq i,j \leq \dimension$. We assume that
\[
\exists \Lambda_A > 0 \ \text{ such that } \sum_{i,j=1}^{\dimension} a_{ij}(x) \xi_i \xi_j \geq \Lambda_A | \xi |^2 \quad \text{ for a.a. } x \in \Omega \ \text{and} \ \forall \xi \in \mathbb R^\dimension.
\]
\subsection{A semilinear elliptic control problem with distributed control}\label{ZZ-S4.1}
Our first problem is
\[
\Pbu \qquad \min_{u\in\uad} J(u): = \mathcal{J}(u)+\frac{\tikhonov}{2}\|u\|_{L^2(\omega)}^2,
\]
where $\tikhonov > 0$, $\uad = \{u\in L^2(\omega)) : \alpha\leq u(x)\leq \beta\text{ for a.a. }x\in\omega)\}$, for some $-\infty\leq\alpha<\beta\leq\infty$, and
\[\mJ(u) = \int_\Omega L(x,y_u(x))\dx,\]
where $y_u\in Y=H^1_0(\Omega)\cap C(\bar\Omega)$ is the unique solution of
\[A y+f(x,y)=\chi_\omega u\text{ in }\Omega,\ y=0\text{ on }\Gamma,\]
where $\chi_\omega$ is the characteristic function of $\omega$, so that $\chi_\omega u$ is the extension by zero of $u$ to $\Omega$.  Here, $f,L:\Omega\times \mathbb{R}\longrightarrow\mathbb{R}$ are Carath\'{e}odory functions of class $C^2$ with respect to the second variable satisfying for almost all $x \in \Omega$:
\begin{itemize}
  \item There exists $\bar p > \dimension/2$ such that $f(\cdot,0)\in L^{\bar p}(\Omega)$,
  \item $\frac{\partial f}{\partial y}(x,y)\geq 0$ for all $y\in\mathbb{R}$,
  \item For all $M>0$ $\exists\,C_{f,M}>0$ such that $\big|\frac{\partial^jf}{\partial y^j}(x,y)\big| \leq C_{f,M}$ for all $|y|\leq M$, $j=1,2$,
    \item For all $M>0$ there exists $K_{f,M}>0$ such that
  \[
  \Big|\frac{\partial^2f}{\partial y^2}(x,y_1)-\frac{\partial^2f}{\partial y^2}(x,y_2)\Big|\leq K_{f,M}|y_1-y_2|   \quad \forall |y_1|,|y_2|\leq M,
  \]
  \item $L(\cdot,0)\in L^1(\Omega)$,
  \item For all $M>0$ $\exists\,\Psi_{L,M}\in L^{\bar p}(\Omega)$ and $\exists\,C_{L,M}>0$ such that $|\frac{\partial L}{\partial y}(x,y)|\leq \Psi_{L,M}(x)$ and $|\frac{\partial^2L}{\partial y^2}(x,y)|\leq C_{L,M}$ for all $|y|\leq M$,
  \item For all $M>0$ there exists $K_{L,M}>0$ such that
  \[
  \Big|\frac{\partial^2L}{\partial y^2}(x,y_1)-\frac{\partial^2L}{\partial y^2}(x,y_2)\Big|\leq K_{L,M}|y_1-y_2|   \quad \forall |y_1|,|y_2|\leq M.
  \]
\end{itemize}

The reader is referred to \cite{CM2024SICON} for the proofs of the next statements.
The mapping $G:L^2(\omega)\longrightarrow Y$, $G(u)=y_u$, is of class $C^2$. For all $u\in L^2(\omega)$ and all $v\in L^2(\omega)$, $z_{u,v} = G'(u)v\in Y$ is the unique solution of
\[Az+\frac{\partial f}{\partial y}(x,y_u)z = \chi_\omega v\text{ in }\Omega,\ z = 0\text{ on }\Gamma.\]
For every $u\in L^2(\omega)$ there exists a unique $\varphi_u\in Y$ solution of the equation
\[A^*\varphi + \frac{\partial f}{\partial y}(x,y_u)\varphi = \frac{\partial L}{\partial y}(x,y_u)\text{ in }\Omega,\ \varphi = 0\text{ on }\Gamma.\]
The mapping $\Phi_Y:L^2(\omega)\to Y$ defined by $\Phi(u) = \varphi_{u}$ is of class $C^1$ and for all $u\in L^2(\omega)$ and all $v\in L^2(\omega)$, $\eta_{u,v}=\Phi_Y'(u)v\in Y$ is the unique solution of
\[A^*\eta+\frac{\partial f}{\partial y}(x,y_u)\eta = \Big[\frac{\partial^2L}{\partial y^2}(x,y_u) - \varphi_u\frac{\partial^2f}{\partial y^2}(x,y_u)\Big] z_{u,v}\text{ in }\Omega,\ \eta = 0\text{ on }\Gamma.\]
We define $\Phi:L^2(\omega)\to L^\infty(\omega)$ as $\Phi(u) = \Phi_Y(u)_{\vert\omega}$.
The mapping $\mJ:L^2(\omega)\to\mathbb{R}$ is of class $C^2$ and for all $u\in L^2(\omega)$ and all $v\in L^2(\omega)$
\[\mJ'(u)v = \int_{\omega} \varphi_u v\dx,\
\mJ''(u)v^2 = \int_{\omega} \eta_{u,v} v\dx.\]

This problem fits in our abstract framework for $X=\omega$ and $\mu$ be the Lebesgue measure, $p=2$. Assumptions \ref{A2.1}, \ref{A2.2} and \ref{A2.3} for $N=0$ are proved in \cite{CM2024SICON}. Assumption \ref{A2.4} is a consequence of the Lipschtiz properties imposed on $L$ and $f$.
The SQP method for this problem is presented in Algorithm \ref{Alg2}.

\LinesNumbered
\begin{algorithm2e}
\caption{SQP method to solve $\Pbu$.}\label{Alg2}
\DontPrintSemicolon
Initialize. Choose $u_0\in B^2_\rho(\bar u)$. Set $n=0$.\;
\For{$n\geq 0$}{
Compute $y_{u_n} = G(u_n)$ solving the nonlinear state equation.
\;
Compute $\varphi_{u_n} = \Phi(u_n)$ solving the linear adjoint state  equation\;
Find a local solution of the constrained quadratic problem
\begin{align*}
\mathrm{(}Q_n\mathrm{)}\quad & \min_{v\in\uad-\{u_n\}} \frac{1}{2}\int_{\omega} \big(\kappa v +\eta_{u_n,v}\big) v\dx + \int_{\omega} (\kappa u_n+ \varphi_{u_n}) v\dx
\end{align*}\;
Name $v_n$ the obtained solution.
\;
Set $u_{n+1}=u_n+v_n$ and $n=n+1$.\;
}
\end{algorithm2e}
Any local minimizer $\bar u$ satisfying the second-order conditions and the strict complementarity condition can be approximated with quadratic order of convergence in the $L^2(\omega)$ and $L^\infty(\omega)$ norms provided an initial point $u_0$ is chosen close enough to $\bar u$ in the $L^2(\omega)$ norm. For every $n\geq 1$, the quadratic program $(Q_n)$ will have a unique local solution in a fixed $L^2(\omega)$ neighborhood of $\bar u$.

\subsection{A semilinear elliptic control problem with distributed and boundary controls}
We will describe this problem for $\dimension = 3$. Let us consider now
\[\Pbd\qquad \min_{\bm u\in\bm U_{\mathrm{ad}}} \mJ(\bm u) + \frac{\tikhonov_1}{2}\|u_1\|^2_{L^2(\Omega)}
+\frac{\tikhonov_2}{2}\|u_2\|^2_{L^2(\Gamma)}.\]
Here $\kappa_i>0$ for $i=1,2$,
\begin{align*}
\bm U_{\mathrm{ad}} = \{\bm u = (u_1,u_2) \in L^2(\Omega)\times L^p(\Gamma):&\ \alpha_1\leq u_1(x)\leq\beta_1\text{ for a.e. }x\in\Omega,\text{ and }\\
&\ \alpha_2\leq u_2(x)\leq\beta_2\text{ for a.e. }x\in\Gamma\},
\end{align*} where $-\infty\leq\alpha_1<\beta_1\leq\infty$, $-\infty<\alpha_2<\beta_2<\infty$ and $p>2$ is a fixed exponent. The functional $\mJ$ is defined as
\[\mJ(\bm u) = \int_\Omega L(x,y_{\bm u}(x))\dx,\]
where $y_{\bm u}\in Y=H^1(\Omega)\cap C(\bar\Omega)$ is the unique solution of
\[A y + f(x,y)=u_1\text{ in }\Omega,\ \partial_{n_A} y = u_2\text{ on }\Gamma.\]
The functions $f$ and $L$ satisfy the same assumptions as in Section \ref{ZZ-S4.1}. We will further assume the existence of a set $E\subset\Omega$ of positive measure such that $a_0(x)>0$ for almost all $x\in E$.

This problem fits into the framework of Section \ref{ZZ-S3.3} for $n=2$, $X_1=\Omega$, $\mu_1$ the $\dimension$-dimensional Lebesgue measure, $X_2=\Gamma$, $\mu_2$ the $(\dimension-1)$-dimensional Lebesgue measure, $p^1=2$, $p^2=p$.

The mapping $G:L^2(\Omega)\times L^p(\Gamma)\longrightarrow Y = H^1(\Omega) \cap C(\bar\Omega)$ defined by $G(\bm u)=y_{\bm u}$ is of class $C^2$.
For all $\bm u, \bm v\in L^2(\Omega)\times L^p(\Gamma)$, $z_{\bm u,\bm v} = G'(\bm u)\bm v\in Y$ is the unique solution of
\[Az+\frac{\partial f}{\partial y}(x,y_{\bm u})z = v_1\text{ in }\Omega,\ \partial_{n_A} z = v_2\text{ on }\Gamma.\]
We remark here that for every $\bm v\in L^2(\Omega)\times L^2(\Gamma)$, the above equation also has a unique solution $z_{\bm u,\bm v}\in H^1(\Omega)$.
For every $\bm u\in L^2(\Omega)\times L^p(\Gamma)$ there exists a unique $\varphi_{\bm u}\in Y$ solution of the adjoint equation
\[
A^*\varphi+\frac{\partial f}{\partial y}(x,y_{\bm u})\varphi = \frac{\partial L}{\partial y}(x,y_{\bm u})\text{ in }\Omega,\ \partial_{n_{A^*}} \varphi = 0\text{ on }\Gamma.
\]
The mapping $\bm\Phi:L^2(\Omega)\times L^p(\Gamma)\longrightarrow Y \times [H^{\frac{1}{2}}(\Gamma) \cap C(\Gamma)] \subset L^\infty(\Omega)\times L^\infty(\Gamma)$ is defined by $\Phi_1(\bm u) = \varphi_{\bm u}$ and $\Phi_2(\bm u) = \mathrm{tr}\varphi_{\bm u}$, where $\mathrm{tr}:H^1(\Omega) \longrightarrow H^{\frac{1}{2}}(\Gamma)$ is the trace operator. For every  $\bm u, \bm v\in L^2(\Omega)\times L^p(\Gamma)$, $\eta_{\bm u,\bm v}\in Y$ is the unique solution of
\[
A^*\eta+\frac{\partial f}{\partial y}(x,y_{\bm u})\eta = \Big[\frac{\partial^2 L}{\partial y^2}(x,y_{\bm u}) - \varphi_{\bm u} \frac{\partial^2 f}{\partial y^2}(x,y_{\bm u})\Big] z_{{\bm u},{\bm v}}\text{ in }\Omega,\ \partial_{n_{A^*}}\eta = 0\text{ on }\Gamma.
\]
We have that $\Phi_1'(\bm u)\bm v = \eta_{\bm u,\bm v}\in Y$ and $\Phi_2'(\bm u)\bm v = \mathrm{tr} \eta_{\bm u,\bm v} \in H^{\frac{1}{2}}(\Gamma) \cap C(\Gamma)$. In this case, notice that for $\bm v\in L^2(\Omega)\times L^2(\Gamma)$, using that $\partial^2_{yy}L(x,y_{\bm u}) - \varphi_{\bm u} \partial^2_{yy} f(x,y_{\bm u})\in L^\infty(\Omega)$ and $z_{{\bm u},{\bm v}}\in H^1(\Omega)\subset L^2(\Omega)$, we do not lose any regularity and $\eta_{\bm u,\bm v}\in Y$. Then, it is easy to deduce that assumptions \ref{A2.2} and \ref{A2.3}, for $p_0=2$ and $p_1=p$ hold. Assumption \ref{A2.4} follows from the Lipschitz properties of $f$ and $L$.

\subsection{The velocity tracking control problem}\label{S3.4}
Let us consider the problem analyzed in \cite{Wachsmuth2007}; see also \cite{Casas-Chrysafinos2016}. In this subsection we assume that $\dimension = 2$ and $\Gamma$ is of class $C^2$. We consider the problem
\[\Pbt\qquad\min_{\bm u\in\uad}J(\bm u) = \int_Q|\bm y_{\bm u}-\bm y_Q|^2\dx\dt+\frac{\tikhonov}{2}\int_{Q_\omega} |\bm u|^2\dx\dt,\]
where $\bm y_Q\in\bm L^{q}(Q)$ for some fixed $2 < q < \infty$ and
\[
\bm \uad = \{\bm u\in \bm L^2(Q_\omega) : \alpha_i\leq u_i\leq\beta_i \text{ for a.e }(x,t)\in Q_\omega\text{ and }i=1,2\}.
\]
Here $-\infty \le \alpha_i<\beta_i \le \infty$ for $i=1,2$, $\tikhonov >0$ and $\bm y_{\bm u}$ is the solution of the evolutionary Navier-Stokes system
\begin{equation}
\left\{\begin{array}{l}\displaystyle\frac{\partial\bm y}{\partial t} - \nu\Delta\bm y + (\bm y \cdot \nabla)\bm y + \nabla\mathfrak{p} = \chi_\omega\bm u \text{ in } Q,\\
\mathrm{div} \bm y = 0 \text{ in } Q, \ \bm y = 0 \text{ on } \Sigma,\ \bm y(0) = \bm y_0 \text{ in } \Omega.
\end{array}\right.
\label{NS}
\end{equation}
For $r\geq 2$, we consider the spaces $\bm L^r(\Omega)=L^r(\Omega)^2$, $\bm W^{1,r}_0(\Omega)= W^{1,r}_0(\Omega)^2$,
$\bm W^{2,r}(\Omega)=W^{2,r}(\Omega)^2$,
\begin{align*}
\bm V= &\{\bm y\in\bm W^{1,2}_0(\Omega):\ \mathrm{div}\bm y=0\},\\
\bm W^{2,1}_r = &\{\bm y\in L^r(0,T;\bm W^{2,r}(\Omega)\cap \bm V) :\partial_t\bm y\in L^r(Q)\}.
\end{align*}
Let $\bm y_0$ be an element of $\bm V$. For every $\bm u\in\bm L^2(Q_\omega)$ there exists a unique solution $\bm y_{\bm u}\in\bm W^{2,1}_2$ of the Navier-Stokes system \eqref{NS}. The mapping $G:\bm L^2(Q_\omega)\longrightarrow \bm W^{2,1}_2$ given by $G(\bm u) = \bm y_{\bm u}$ is of class $C^\infty$. For all $\bm u, \bm v\in \bm L^2(Q_\omega)$, $\bm z_{\bm u,\bm v}=G'(\bm u)\bm v \in \bm W^{2,1}_2$ is the unique solution of
\[
\left\{\begin{array}{l}\displaystyle\frac{\partial\bm z}{\partial t} - \nu\Delta\bm z + (\bm y \cdot \nabla)\bm z + (\bm z \cdot \nabla)\bm y + \nabla\mathfrak{q} = \chi_\omega\bm v \text{ in } Q,\\
\mathrm{div} \bm z = 0 \text{ in } Q, \  \bm z = 0 \text{ on } \Sigma,\ \bm z(0) = \bm 0 \text{ in } \Omega.\end{array}\right.
\]

For every $\bm u\in\bm L^2(Q_\omega)$ there exists a unique $\bm\varphi_{\bm u}\in \bm W^{2,1}_{q}$ solution of the adjoint state equation
\[
\left\{\begin{array}{l}\displaystyle -\frac{\partial\bm\varphi}{\partial t} - \Delta\bm\varphi - (\bm y \cdot \nabla)\bm\varphi - (\nabla\bm\varphi)^T\bm y + \nabla\pi = \bm y_{\bm u} - \bm y_d\text{ in } Q,\\
\mathrm{div} \bm\varphi = 0 \text{ in } Q, \  \bm\varphi = 0 \text{ on } \Sigma,\ \bm\varphi(T) = \bm 0 \text{ in } \Omega.\end{array}\right.
\]
The $\bm W^{2,1}_{q}$ regularity of $\bm\varphi_{\bm u}$ follows from the fact that $\bm y_{\bm u} - \bm y_d \in \bm L^{q}(Q)$. This is assumed for $\bm y_d$ and it is satisfied by $\bm y_{\bm u}$ due to the inclusion $\bm W^{2,1}_2 \subset L^\infty(0,T;\bm H^1(\Omega))$; see \cite[Theorem III.4.10.2]{Amann1995}. Let us remark that $\bm W^{2,1}_2\not\subset\bm L^\infty(Q)$ but $q>2$ implies that $\bm W^{2,1}_{q}\subset C([0,T];W^{1,p}(\Omega)) \subset \bm C(\bar Q)$; \cite[Theorem III.4.10.2]{Amann1995}.

For all $\bm u, \bm v \in \bm L^2(Q_\omega)$, there exists a unique $\bm\eta_{\bm u,\bm v}  \in \bm W^{2,1}_{q}$ solution of
\[
\left\{\begin{array}{l}\displaystyle -\frac{\partial\bm\eta}{\partial t} - \nu\Delta\bm\eta - (\bm y \cdot \nabla)\bm\eta - (\nabla\bm\eta)^T\bm y + \nabla\phi = (\bm z_{\bm u,\bm v}\cdot\nabla)\bm\varphi_{\bm u} + (\nabla\bm\varphi_{\bm u})^T\bm z_{\bm u,\bm v}\text{ in } Q,\\
\mathrm{div} \bm\eta = 0 \text{ in } \Omega, \  \bm\eta = 0 \text{ on } \Gamma,\ \bm\eta(T) = \bm 0 \text{ in } \Omega.\end{array}\right.
\]
To check this it is enough to prove that the right hand side of the above pde belongs to $\bm L^{q}(Q)$. To this end we first observe that $\bm z_{\bm u,\bm v} \in L^{q}(0;T;L^\infty(\Omega))$. This is a straightforward consequence of \cite[Theorem 3]{Amann2001} taking $X_0 = \bm L^2(\Omega)$, $X_1 = \bm H^2(\Omega)$, $s < \frac{1}{2}$, $\theta > \frac{1}{2}$, and $s$ and $\theta$ close enough to $\frac{1}{2}$.  Then, we have
\begin{align*}
&\int_Q|\bm z_{\bm u,\bm v}|^{q}|\nabla\bm\varphi_{\bm u}|^{q}\dx \dt \le \int_0^T\|\bm z_{\bm u,\bm v}(t)\|^{q}_{L^\infty(\Omega)}\int_\Omega|\nabla\bm\varphi_{\bm u}|^{q}\dx \dt\\
&\le \int_0^T\|\bm z_{\bm u,\bm v}(t)\|^{q}_{L^\infty(\Omega)}\dt\|\nabla\bm\varphi_{\bm u}\|^{q}_{L^\infty(0,T;L^{q}(\Omega))} < \infty.
\end{align*}
In our framework, $X=Q_\omega$, $\mu$ is the Lebesgue measure and $p=2$. Since $\bm W^{2,1}_{q} \subset \bm L^\infty(Q)$, the mapping $\bm\Phi:\bm L^2(Q_\omega)\longrightarrow \bm L^\infty(Q_\omega)$ given by $\bm\Phi(\bm u)=\bm\varphi_{\bm u\vert Q_\omega}$ is of class $C^\infty$ and for all $\bm u, \bm v \in \bm L^2(Q_\omega)$, $\bm\Phi'(\bm u)\bm v  =\bm\eta_{\bm u,\bm v\vert Q_\omega}$.
The assumption \eqref{E2.2} follows from the compactness of the embedding $\bm W^{2,1}_2\subset\bm L^2(Q)$ and the generalized mean value theorem.
The assumption \ref{A2.3} is satisfied with $p_0=p=2$. Assumption \ref{A2.4} follows from the generalized mean value theorem and the fact that $\bm\Phi:\bm L^2(Q_\omega)\longrightarrow  \bm L^\infty(Q_\omega)$ is of class $C^2$.

\subsection{A semilinear parabolic bilinear control problem with boundary control}\label{S3.5}
For a complete analysis of the following problem te reader is referred to \cite{CM2025}.
We focus on the problem
\[\Pbc\qquad\min_{u\in\uad}J(u) = \mJ(u)+\frac{\tikhonov}{2}\int_\Sigma u(x,t)^2\dx\dt,\]
where $\uad = \{u\in L^p(\Sigma):\umin\leq u(x,t)\leq\umax\text{ for a.e }(x,t)\in \Sigma\}$ with $p = 2(\dimension+1)$ and $0\leq \alpha<\beta <\infty$.
Here, $\mJ(u) = \frac{1}{2}\|y_u-y_d\|^2_{L^2(Q)}$ where $y_d\in L^{\tilde p}(Q)$ for some $\tilde p > 1+\dimension/2$ and $y_u\in Y=W(0,T)\cap L^\infty(Q)$ is the solution of the equation
\begin{equation}
\left\{\begin{array}{l} \displaystyle\frac{\partial y}{\partial t} + Ay + f(x,t,y) = 0\ \  \mbox{in } Q,\vspace{2mm}\\  \partial_{\conormal_A} y + uy = g\ \ \mbox{on }\Sigma, \ y(x,0) = y_0(x) \ \ \text{in } \Omega. \end{array}\right.
\label{E4.2}
\end{equation}
We assume that $g\in L^{\bar p}(\Sigma)$ for some $\bar p>\dimension +1$, $y_0\in L^\infty(Q)$ and $f:Q\times \mathbb{R}\longrightarrow \mathbb{R}$ is a Carath\'{e}odory function of class $C^2$ w.r.t. the second variable such that for almost all $(x,t) \in Q$
\begin{itemize}
  \item $f(\cdot,\cdot,0)\in L^{\tilde p}(Q)$.
  \item $\exists\, C_f\in\mathbb{R}$ such that $\frac{\partial f}{\partial y}(x,t,y)\geq C_{f}$ for all $y\in\mathbb{R}$.
  \item $\forall M>0$ $\exists\,C_{f,M}>0$ such that $\Big|\frac{\partial^jf}{\partial y^j}(x,t,y)\Big|\leq C_{f,M}$ $\forall |y|\leq M$, $j=1,2$.
  \item $\forall M>0$ $\exists\,K_{f,M}>0$ such that $\Big|\frac{\partial^2 f}{\partial y^2}(x,t,y_1)-\frac{\partial^2 f}{\partial y^2}(x,t,y_2)\Big|\leq K_{f,M}|y_1-y_2|$ $\forall |y_1|,|y_2|\leq M$.
\end{itemize}

Under the previous assumptions there exists an open set $\mathcal{A}\subset L^p(\Sigma)$ such that $\uad\subset\mA$ and for every $u\in\mA$
the mapping $G:\mA\longrightarrow Y$, given by $G(u) = y_u$ solution of \eqref{E4.2}, is well defined and of class $C^2$. For all $u\in \mA$ and all $v\in L^p(\Sigma)$, $z_{u,v}=G'(u)v\in Y$ is the unique solution of the equation
\[
\frac{\partial z}{\partial t} + A z + \frac{\partial f}{\partial y}(x,t,y_u)z=0\text{ in }Q,\ \partial_{n_A} z + u z=-v y_u\text{ on }\Sigma,\ z(0)=0\text{ in }\Omega.
\]
It is important to notice that for $v\in L^2(\Sigma)$, the above equation has a unique solution in $W(0,T)$ that we will also name $z_{u,v}$. Furthermore, we have that ${z_{u,v}}_{|\Sigma}\in L^{2+2/\dimension}(\Sigma)$.

For every $u\in \mA$ there exists a unique $\varphi_u\in Y$ solution of
\[
-\frac{\partial\varphi}{\partial t} + A^* \varphi + \frac{\partial f}{\partial y}(x,t,y_u)\varphi=y_u-y_d\text{ in }Q,\ \partial_{n_A^*}\varphi + u \varphi=0\text{ on }\Sigma,\ \varphi(T)=0\text{ in }\Omega.
\]
The mapping $\Psi:\mA\to Y$ given by $\Psi(u)=\varphi_u$ is of class $C^1$ and for all $u\in \mA$ and all $v\in L^p(\Sigma)$, $\eta_{u,v}=\Psi'(u)v\in Y$ is the unique solution of the equation
\[
\left\{\begin{array}{l}-\frac{\partial\eta}{\partial t} + A^*\eta+ \frac{\partial f}{\partial y}(x,t,y_u)\eta = \Big[1 - \varphi_u \frac{\partial^2f}{\partial y^2}(x,t,y_u)\Big] z_{u,v}\text{ in }Q,\\
\partial_{n_A^*}\eta +u\eta= -v\varphi_u \text{ on }\Sigma,\qquad \eta(T)=0\text{ on }\Omega.\end{array}\right.
\]
For every $v\in L^2(\Sigma)$, the above equation has a unique solution in $W(0,T)$ that we will also name $\eta_{u,v}$. Furthermore, we have that $\eta_{u,v}\in L^{2+2/\dimension}(\Sigma)$.

The mapping $\mJ:\mA\longrightarrow \mathbb{R}$ is of class $C^2$ and for all $u\in\mA$ and all $v\in L^p(\Sigma)$,
\[\mJ'(u) v=\int_\Sigma(-y_u\varphi_u) v\dx\dt,\qquad \mJ''(u) v^2 = \int_\Sigma(-y_u\eta_{u,v}-z_{u,v}\varphi_u) v\dx\dt.\]

This problem fits into our framework for $X=\Sigma$, $\mu$ equal to the $\dimension$ dimensional Lebesgue measure on $\Sigma$, and $p = 2(\dimension +1)$. The $C^1$ mapping $\Phi:\mA\longrightarrow L^\infty(\Sigma)$ is given by $\Phi(u) = (-y_u\varphi_u)_{|\Sigma}$. We have that $\Phi'(u)v=(-y_u\eta_{u,v}-z_{u,v}\varphi_u)_{|\Sigma}$ for all $u\in\mA$ and all $v\in L^p(\Sigma)$.

Using that $y_u,\varphi_u\in L^\infty(Q)$, $\Phi'(u)$ can be extended to a linear and continuous mapping $\Phi'(u):L^2(\Sigma) \longrightarrow L^2(\Sigma)$. The compactness of this operator follows from the compactness of the trace mapping from $W(0,T)$ into $L^2(\Sigma)$; see \cite[Theorem A.2]{CM2025}. The inequality \eqref{E2.4} is deduced from \cite[Lemma 4.2]{CM2025} and the fact that $p=2(d+1)$. Hence, Assumption \ref{A2.2} holds.

From \cite[Theorem A.3]{CM2025}, 
we deduce that $\Phi'(u)v\in L^{p_i}(\Sigma)$ if $v\in L^{p_{i-1}}(\Sigma)$, for $p_0=2$, $p_1= 3$, $p_2=6=p$ if $\dimension =2$
and $p_0=2$, $p_1 = 8/3$, $p_2= 5$, $p_3=8=p$ if $\dimension = 3$,
and therefore Assumption \ref{A2.3} is satisfied.
 The Lipschitz property of $\Phi'$ required in Assumption \ref{A2.4} follows from the Lipschitz property of $\frac{\partial^2 f}{\partial y^2}$; see \cite[Theorem 3.5]{CM2025}.

The SQP method for this problem is described in Algorithm \ref{Alg3}. Notice that, since $\uad$ is bounded in $L^\infty(\Sigma)$, we can choose $u_0\in\uad\cap B^2_{\rho'}(\bar u)$ for some appropriate $\rho'>0$.

\LinesNumbered
\begin{algorithm2e}
\caption{SQP method to solve $\Pbc$.}\label{Alg3}
\DontPrintSemicolon
Initialize. Choose $u_0\in B^2_r(\bar u)\cap\uad$. Set $n=0$.\;
\For{$n\geq 0$}{
Compute $y_{u_n} = G(u_n)$ solving the nonlinear state equation.
\;
Compute $\varphi_{u_n} = \phi(u_n)$ solving the linear adjoint state equation\;
Find a local solution of the constrained quadratic problem
\begin{align*}
\mathrm{(}Q_n\mathrm{)}\quad & \min_{v\in\uad-\{u_n\}} \frac{1}{2}\int_\Sigma \big(\kappa v -(\varphi_{u_n}z_{u_n,v} + y_{u_n}\eta_{u_n,v})\big) v\dx\dt \\
& \qquad \qquad + \int_\Sigma (\kappa u_n- \varphi_{u_n}y_{u_n}) v\dx\dt
\end{align*}\;
Name $v_n$ the obtained solution.
\;
Set $u_{n+1}=u_n+v_n$ and $n=n+1$.\;
}
\end{algorithm2e}

\section{Numerical examples and discussion}\label{ZZ-S5}

In \cite{Troltz1999}, an SQP method for solving control-constrained control problems is analyzed. Unlike the method presented in this paper, in \cite{Troltz1999} the three variables -- state, adjoint state, and control -- are treated as independent variables related by the state and adjoint state equations, which are viewed as equality constraints. We will refer to this method as \textsc{sqplin}, because only linear PDEs need to be solved, and to our method as \textsc{sqpnln}, because our method requires solving one nonlinear PDE per iteration.

A natural question is whether one of the methods is superior to the other. Since both exhibit quadratic order of convergence, it may seem that \textsc{sqplin}, which requires only the solution of linear equations, should be faster in practice. However, other practical issues can affect the performance of the algorithms. In particular, choosing a good starting point may be more difficult for \textsc{sqplin} than for \textsc{sqpnln}.

Without any pretension of claiming a superior performance of any of the two methods, we compare the speed of execution for two of the examples studied above: $\Pbu$ --elliptic equation, additive distributed control-- and $\Pbc$ --parabolic equation, bilinear boundary control--.

We solve finite element approximations of the problems.
For the space discretization we use a uniform mesh of size $h=2^{-N}$, for some appropriate $N$. Continuous piecewise linear elements are employed for the state and the adjoint state. For $\Pbu$ the control is discretized using piecewise constant elements.
For the time-dependent problem $\Pbc$ we use a discontinuous Galerkin scheme in time, equivalent to the implicit Euler method, with step-size $\tau=2^{-N} T$. In this case the discretization of the control in space is done with continuous piecewise linear functions, and the Tichonov regularization is done using an appropriate mass lumping in space; see \cite{CasasMateosRosch2018} for a detailed explanation of this discretization technique.

The algorithms stop when the maximum of the absolute and relative error between two iterations, measured in the $L^\infty$ norm, is less than $\varepsilon=5\times 10^{-13}$ or when the last two objective values are equal up to machine precision. When both algorithms succeed, we have observed that the relative difference (measured in $L^\infty$) of the obtained solutions is less than $\varepsilon$.

All the software has been programmed by us with \textsc{Matlab} and has been run on a desktop PC with 32GB of RAM and an Intel i7-13700F processor. The reported times are obtained by a simple \texttt{tic-toc} measure. The quadratic problems are solved using a semismooth Newton method. The unconstrained quadratic program that appears at each iteration of the semismooth Newton method is solved using Matlab's implementation \texttt{pcg} of the conjugate gradient method. The nonlinear partial differential equations are solved using Newton's method taking the value of the current state $y$ in memory as initial point, except in the first call, where we take $y=0$. All the linear systems arising from the finite element method are solved using Matlab's {\tt chol} and the backslash. Great careful has been taken in order to avoid factorizing more than once the same matrix. The nonlinear terms are integrated using the 10-point formula for tetrahedrons described in \cite{Shun-Ham2012}.

\begin{example}\label{Ex4.1}
\textit{A semilinear elliptic control problem with distributed control.}
For the problem $\Pbu$ described in Section \ref{ZZ-S4.1}, consider $\dimension = 3$, $\Omega=(0,1)^\dimension$, $Ay = -\Delta y$, $f(x,y)=\mathrm{e}^y$, $L(x,y) = 0.5|y-y_d(x)|^2$, with $y_d(x) = \prod_{i=1}^\dimension 8 x_i(1-x_i)$, $\kappa =0.1$, $\alpha = 0.1$ and $\beta = 1$. The discretization is done with $N=6$ uniform refinements for the cube, i.e., $h=1/64$. We have $1.572.864$ degrees of freedom (dof) for the control and $250.047$ both for the state and the adjoint state.

The initial point is $u_0=(\umin+\umax)/2$ and, for \textsc{sqplin}, $y_0=0$ and $\varphi_0=0$. Our algorithm converges in three iterations and the whole process takes 199 seconds.  The method that we denote \textsc{sqplin} takes four iterations, but only 122 seconds.
\end{example}

\begin{example}\label{Ex4.2}
\textit{A semilinear parabolic bilinear control problem with boundary control.}
The SQP introduced in \cite{Troltz1999} is studied in detail in \cite{CM2025} for problem $\Pbc$. We solve the problem presented in Section 6 of that reference.

We consider $\Omega=(0,1)^\dimension$ for $\dimension = 3$, and $T=4$. The data for the equation are $Ay = -\Delta y$, $f(x,t,y) = y^3-y$, $y_0(x) = \prod_{i=1}^{d}8 x_i(1-x_i)$, and $g(x,t) = 1$. To define the objective functional,
$y_d(x,t) = y_0(x)\cos\left(\pi t\right)$,
and  $\kappa = 0.3$. The control constraints are given by $\umin = 0.1$ and $\umax = 100$.

The discretization is done using $N=5$ uniform refinements of the cube, so the last spatial mesh size is $h=1/32$ and the last time step is $\tau=0.125$. We have $196.672$ degrees of freedom (dof) for the control and $1.149.984$ dof for both the state and the adjoint state.

We use as initial point $u_0=(\umin+\umax)/2$. Our algorithm converges in 6 iterations to the optimal value (about $100$ seconds if $N=4$ and $5\times 10^3$ seconds if $N=5$). The convergence history is shown in Table \ref{Tabla1}. We include the value of the functional, the stepsize measured as $\delta_n = \|v_{n-1}\|_{L^\infty(\Sigma)}/\max\{1,\|u_n\|_{L^\infty(\Sigma)}\}$, and the number of nodes where the constraints are inactive and active.

{\small
\begin{table}
  \centering
  \caption{Convergence history for Algorithm \ref{Alg3}. Problem $\Pbc$.}\label{Tabla1}
\[
\begin{array}{cccccc}
 n &   J (u_n)                 &\delta_n      &\sharp\{\umin<\bar u<\umax\} &  \sharp\{\bar u=\umin\}  & \sharp\{\bar u=\umax\}\\ \hline
  0 &   9.0274091266354717e+03 &              & 196672  &     0    &       0\\
  1 &   1.6728953004109695e+01 &    5.0e+01   &      0  &   196672 &      0 \\
  2 &   1.3529647576662601e+01 &    9.4e-01   & 166942  &    29730 &      0 \\
  3 &   1.3441235676498732e+01 &    2.1e-01   & 165604  &  31068   &      0 \\
  4 &   1.3441100623640869e+01 &    8.4e-03   & 165580  &   31092  &      0 \\
  5 &   1.3441100623224251e+01 &    2.0e-05   & 165580  &   31092  &      0 \\
  6 &   1.3441100623224251e+01 &    1.6e-10   & 165580  &   31092  &      0
\end{array}
\]
\end{table}
}

In contrast, the \textsc{spqlin} algorithm does not converge for $u_0=(\umin+\umax)/2$, $y_0=\varphi_0=0$. The strategy of taking $y_0=y_{u_0}$ and $\varphi_0=\varphi_{u_0}$ does not work for this choice of $u_0$ either. A better initial point is needed. Since we have found the solution with our method, we know that the average value of $\bar u$ is close to $0.6$. So we tried $u_0=0.6$. Again, $y_0=\varphi_0=0$ does not provide a convergent sequence, but for this choice of $u_0$, if we solve the nonlinear state equation and the adjoint state equation once and choose $y_0=y_{u_0}$ and $\varphi_0=\varphi_{u_0}$, we get convergence in 6 iterations (96 seconds with $N=4$ refinements and 4664 seconds with $N=5$ refinements). If we use our method with $u_0=0.6$, the method converges in 5 iterations (92 seconds if $N=4$, 4507 seconds if $N=5$). A more sophisticate search for an initial point for \textsc{SQPLIN} is described in \cite{CM2025}.
\end{example}

\paragraph{Conclusion} Numerical experience shows that the \textsc{spqnln} algorithm can be more computationally efficient than \textsc{spqlin}, which may even fail for convergence where \textsc{spqnln} does not. This is especially true when the state equation is evolutionary.

\appendix
\section{Lipschitz stability in $L^2(X)$}\label{S2.4}
We prove below Lemma \ref{L2.16} used to obtain \eqref{ZZ-E3.17}. This result can be seen as a strong regularity property or Lipschitz stability of solutions of constrained quadratic programs, cf. \cite{Troltz2000}, \cite[Section 4]{Troltz1999}, \cite[Section 5.1]{Hoppe-Neitzel2021}, \cite[Section 5.3]{Hehl-Neitzel2024}. Indeed, this is the analogous of \cite[Theorem 4.2]{Troltz2000}.
\begin{lemma}\label{L2.16}
Let $a:L^2(X)\times L^2(X) \longrightarrow \mathbb{R}$ be a continuous, symmetric, and bilinear form such that there exists $\lambda>0$ satisfying
\[
a(w,w)\geq \lambda\|w\|_{L^2(X)}^2\quad \forall w\in E^\tau_{\bar u}.
\]
Let $b_0$ and $b_1$ be two elements of $L^2(X)$ and
for $i=0,1$, let $w_i \in \widehat{\uad}-\{\bar u\}$ be the solution of the constrained quadratic program
\[(Q_i)\qquad\min_{w\in \widehat{\uad}-\{\bar u\}} \frac{1}{2}a(w,w) + \int_X b_i w\mathrm{d}x.\]
Then, we have
\[
\|w_1-w_0\|_{L^2(X)}\leq \frac{1}{\lambda} \|b_1-b_0\|_{L^2(X)}.
\]
\end{lemma}
\begin{proof}
  First of all, from the continuity and the coercivity property of the bilinear form and the convexity of $\widehat{\uad}-\{\bar u\}$, it is clear that each problem $(Q_i)$ has a unique solution $w_i\in \widehat{\uad}-\{\bar u\}$. First order optimality conditions for each of these problems reads like
  \begin{eqnarray*}
    a(w_i,w-w_i) + \int_X b_i(w-w_i)\mathrm{d}x &\geq & 0\ \forall w\in \widehat{\uad}-\{\bar u\},\ \ i=0, 1.
  \end{eqnarray*}
  Testing the inequality satisfied by $w_i$ with $w_{1-i}\in \widehat{\uad}-\{\bar u\}$ , we obtain
  \begin{eqnarray*}
        a(w_1,w_0-w_1) +  \int_X b_1(w_0-w_1)\mathrm{d}x &\geq & 0\\
        a(w_0,w_1-w_0) +  \int_X b_0(w_1-w_0)\mathrm{d}x &\geq & 0\\
  \end{eqnarray*}
  Adding this inequalities and grouping and reordering the terms, we obtain
  \[a(w_1-w_0,w_1-w_0)\leq  \int_X (b_1-b_0)(w_0-w_1)\mathrm{d}x.\]
  Since $w_0 = w_1=0$  a.e. $[\mu]$ in $X^\tau_{\bar u}$, we deduce that $w_1-w_0\in E^\tau_{\bar u}$.
  Using the coercivity of the bilinear form on $E^\tau_{\bar u}$ and the Schwarz inequality, the result follows immediately.
\end{proof}


\end{document}